\documentclass[12pt, a4paper, english]{article}		
\usepackage[T1]{fontenc}
\usepackage{babel}		
\usepackage{fancyhdr}		
\usepackage[utf8]{inputenc}
\usepackage{amsmath}  
\usepackage{amsthm}		
\usepackage{amsfonts}		
\usepackage{amssymb} 
\usepackage{stmaryrd} 
\usepackage{enumerate} 
\usepackage{graphicx} 
\usepackage{placeins}
\usepackage{indentfirst}
\usepackage{url}
\usepackage{xcolor}
\newtheorem{teorema}{Theorem}
\newtheorem{lema}[teorema]{Lemma}
\newtheorem{proposicio}[teorema]{Proposition}
\newtheorem{corolari}[teorema]{Corollary}

\theoremstyle{definition}
\newtheorem{definicio}[teorema]{Definition}
\newtheorem{nota}[teorema]{Remark}
\newtheorem{exemple}[teorema]{Example}
\newtheorem{aex}{Example}

\newtheorem{problema}[teorema]{Problem}

\numberwithin{teorema}{section}

\DeclareMathOperator{\C}{C}

\DeclareMathOperator{\N}{N}

\DeclareMathOperator{\Aut}{Aut}

\DeclareMathOperator{\Ker}{Ker}

\DeclareMathOperator{\Syl}{Syl}
\DeclareMathOperator{\Fix}{Fix}
\DeclareMathOperator{\Fit}{Fit}
\DeclareMathOperator{\Soc}{Soc}

\DeclareMathOperator{\Frat}{Frat}
\DeclareMathOperator{\Z}{Z}

\title{
\vspace{-4cm}
Central nilpotency of left skew braces and solutions of the Yang--Baxter equation}

\author{A. Ballester-Bolinches\thanks{Departament de Matem\`atiques, Universitat de Val\`encia, Dr.\ Moliner, 50, 46100 Burjassot, Val\`encia, Spain, \texttt{Adolfo.Ballester@uv.es}, \texttt{Ramon.Esteban@uv.es}, \texttt{Vicent.Perez-Calabuig@uv.es}}\  \and R. Esteban-Romero\addtocounter{footnote}{-1}\footnotemark \and M. Ferrara\thanks{Dipartimento di Matematica e Fisica, Università degli Studi della Campania  ``Luigi Vanvitelli'', viale Lincoln 5, 81100 Caserta, Italy, \texttt{maria.ferrara1@unicampania.it}} \and V. P\'erez-Calabuig\addtocounter{footnote}{-2}\footnotemark \and M. Trombetti\addtocounter{footnote}{1}\thanks{Dipartimento di Matematica e Applicazioni ``Renato Caccioppoli'', Università degli Studi di Napoli Federico II, Complesso Universitario Monte S. Angelo, Via Cintia, 80126 Napoli, Italy, \texttt{marco.trombetti@unina.it}}}

\date{}
\begin{document}
\maketitle

\begin{abstract}
\noindent Nipotency of skew braces is related to certain  types of solutions of the Yang--Baxter equation. This paper delves into the study of centrally  nilpotent skew braces. In particular, we study their torsion theory (Section~\ref{sec:tor}) and we introduce an ‘‘index’’ for subbraces (Section~\ref{sec:index}), but we also show that the product of centrally nilpotent ideals need not be centrally nilpotent (Example \ref{ex:nofitting}), a rather peculiar fact. To cope with these examples, we introduce a special type of nilpotent ideal, using which, we define a {\it good} Fitting ideal. Also, a Frattini ideal is defined and its relationship with the Fitting ideal is investigated.

A key ingredient in our work is the characterisation of the commutator of ideals in terms of absorbing polynomials (Section \ref{sec:unif}); this solves Problem 3.4 of \cite{BonattoJedlicka22}. Moreover, we provide an example (Example~\ref{ex:noidealiser}) showing that the idealiser of a subbrace (as defined in \cite{JespersVanAntwerpenVendramin22}) does not exist in general.
\end{abstract}

\noindent\emph{Mathematics Subject Classification \textnormal(2022\textnormal)}:  Primary 16T25, 20F18\\ \phantom{\emph{Mathematics Subject Classification \textnormal(2022\textnormal)}:  }Secondary 81R50

\smallskip

\phantom{\emph{Keywords:}} commutator of ideals, idealiser

\section{Introduction}

\noindent The study of set-theoretical solutions of the Yang--Baxter equation (YBE) provides a common framework for a multidisciplinary approach from different  areas including knot theory, braid theory, or Garside theory among others (see \cite{Chouraqui10}, \cite{Faddev}, \cite{Gateva-Ivanova18-advmath},  for example). 

A (finite) \emph{set-theoretical solution} of the YBE is a pair $(X,r)$, where $X$ is a (finite) set and $r\colon X \times X \rightarrow X \times X$ is a bijective map satisfying the equality $r_{12}r_{23}r_{12} = r_{23}r_{12}r_{23}$, where $r_{12} = r\times \operatorname{id}_X$ and $r_{23} = \operatorname{id}_X \times\, r$. The main challenge in this area is to classify all set-theoretical solutions with prescribed properties. The algebraic structure of left skew braces plays a fundamental role in this classification.

A \emph{left skew brace} $B$ is a set endowed with two group structures, $(B,+)$ and $(B,\cdot)$, satisfying the following {\it skew} distributivity property $$a \cdot (b+c) = a\cdot b - a + a\cdot c\quad \forall\, a,b,c\in B.$$ If $(B,+)$ satisfies some property $\mathfrak{X}$ (such as abelianity), we say that $B$ is a left skew brace of {\it $\mathfrak{X}$ type}. Of course, left skew braces of abelian type are just  Rump's \emph{left braces} (see~\cite{GuarnieriVendramin17} and~\cite{Rump07}).

A {\it non-degenerate} set-theoretic solution of the YBE (solution, for short), i.e. a solution for which both components of $r$ are bijective, naturally leads to a left skew brace structure over the group 
\[G(X,r) = \langle x \in X\,|\, xy = uv,\, \text{if $r(x,y) = (u,v)$}\rangle\] (see \cite{Rump07}) --- this is said to be the \emph{structure left skew brace of~$(X,r)$} (note that it is an infinite skew brace). Conversely, every left skew brace $B$ defines a solution $(B,r_B)$ of the YBE (see~\cite{GuarnieriVendramin17}).

Nowadays, we are far from being able to understand arbitrary solutions of the YBE. But it is becoming more and more clear that almost every solution is multipermutation --- recall that a {\it multipermutation} solution is one that can be retracted into the trivial solution over a singleton after finitely many identification steps --- and nilpotency of left skew braces is precisely introduced to deal with multipermutation solutions (see \cite{cedo,cameron,JespersVanAntwerpenVendramin22} for example). In this paper, we go through a deep study of central nilpotency of left skew braces, with the aim of providing a rigorous framework that could be used to show that almost every solution is multipermutation. We hope our paper could be a reference point for all future work on the argument. 

\smallskip

We now highlight some of the main aspects of nilpotency of left skew braces we deal with (see next sections for the definitions):

\begin{itemize}
    \item In Section \ref{sec:tor}, we study its torsion theory, providing great extensions of some of the results in \cite{JespersVanAntwerpenVendramin22} (see in particular Theorem \ref{moltobello}).
    \item In Section \ref{sec:index}, we deal with the problem of defining an index for subbraces, proving that this is possible in the context of locally centrally-nilpotent left skew braces.
\end{itemize}

\medskip

Also we provide many examples showing how peculiar is the behaviour of centrally nilpotent left skew braces if compared to that of groups and rings. Among these examples, two are the most striking (see Section \ref{wex}):

\begin{itemize}
    \item Example \ref{ex:noidealiser} shows that, contrary to what claimed in \cite{JespersVanAntwerpenVendramin22}, it is not always possible to define the idealiser of a subbraces of a left skew brace of abelian type.
    \item Example \ref{ex:nofitting} shows that the product of two centrally nilpotent ideals of a left skew brace (of abelian type) need not be centrally nilpotent.
\end{itemize}

\medskip

In order to cope with the above examples, we introduce a new type of nilpotency for ideals (see Section \ref{sec:bnilp}). Using this new concept we are able to define a well-behaving {\it Fitting ideal} for left skew braces. In turns, the ~Fit\-ting ideal makes it possible to give a good definition of {\it Frattini ideal} for left skew braces. Using these definitions we are able to prove an analogue of a celebrated result of Gasch\"utz that relates the Frattini and the Fitting subgroups of a group (see Theorem \ref{gaschutz}).

\medskip

Finally, we note that one of the key steps in our work is showing that the two known definitions of commutators of ideals (inspired by distinct Universal Algebra approaches) coincide, thus solving Problem 3.4 of \cite{BonattoJedlicka22} (see Section \ref{sec:unif}).

\section{Preliminaries}
\label{sec:prelim}

\emph{From now on, the word ‘‘brace’’ will mean ‘‘left skew brace’’}

\smallskip

\noindent In this section, we fix notation and give some background concepts and results on braces. Although some of them can be considered folklore, others are new.

Let $(B,+,\cdot)$ be a brace. Both operations in $B$ can be related by the so-called {\it star product}: $a\ast b = -a + ab -b$, for all $a,b\in B$. Indeed, both group operations coincide if and only if $a\ast b=0$ for all $a,b\in B$; in this case, $B$ is said to be a \emph{trivial brace}. The following properties of the star product are essential to our work:
\begin{align}
(ab) \ast c & =  a \ast (b\ast c) + b\ast c + a \ast c;\label{eq:identitat_astprod}\\
ab & =  a + a \ast b + b; \label{eq:prod_ast}\\
a \ast (b+c) & = a \ast b + b + a \ast c - b,\label{eq:identitat_astsum}
\end{align}
for all $a,b,c\in B$. If $X$ and $Y$ are subsets of $B$, then $X \ast Y$ is the subgroup of~$(B,+)$ generated by the elements of the form $x \ast y$, for all $x \in X$ and $y \in Y$.

For every brace, $0$ denotes the common identity element of both group operations, and the product of two elements will be denoted by juxtaposition. As usual, group addition follows group product in the order of operations; also, the star product comes first in the order of operations.

A \emph{subbrace} of a brace is a subgroup of the additive group which is also a subgroup of the multiplicative group. A \emph{homomorphism} between two braces $A$ and $B$ is a map $f\colon A \rightarrow B$ satisfying that $f(a + b) = f(a) + f(b)$ and $f(ab) = f(a)f(b)$ for all  $a,b\in A$. Two braces are said to be \emph{isomorphic} if there exists a bijective homomorphism between them.

Let $B$ denote an arbitrary brace. Given two subsets $X,Y \subseteq B$, we write $\langle X \rangle_+,  [X,Y]_+$ for the subgroup generated by $X$ and the commutator of $X$ and $Y$ in $(B,+)$, respectively, and we write $\langle X\rangle_\cdot, [X,Y]_\cdot$ for the subgroup generated by $X$ and the commutator of $X$ and $Y$ in $(B,\cdot)$, respectively. Similarly, if we need to emphasize that some expressions are related to either the additive or the multiplicative structure, then we put a $+$ or $\cdot$ symbol, respectively, next to it.

The so-called $\lambda$-\emph{action} describes an action of the multiplicative group of~$B$ on the additive one. For every $a\in B$, the map $\lambda_a \colon B \rightarrow B$, given by $\lambda_a(b) = -a + ab$, is an automorphism of $(B,+)$ and the map $\lambda\colon (B,\cdot) \rightarrow \Aut(B,+)$ which sends $a \mapsto \lambda_a$ is a group homomorphism. For every $a,b\in B$,
\[ a \ast b := \lambda_a(b) - b = - a + ab - b.\] \indent The following concepts are apt to describe the substructural framework of an arbitrary brace $B$. \emph{Left ideals} are $\lambda$-invariant subbraces, or equivalently subbraces $L$ such that $B \ast L \subseteq L$. A left ideal $S$ is said to be a \emph{strong left ideal} if $(S,+)$ is a normal subgroup of $(B,+)$, and an \emph{ideal} if $(S,\cdot)$ is also a normal subgroup of $(B,\cdot)$, or equivalently $S \ast B \subseteq S$. Ideals of braces can be considered as the true analogues of normal subgroups in groups and ideals in rings, since they allow to take quotients in a brace. Thus, if $I$ is an ideal of $B$, then $B/I = \{bI = b+I: b \in B\}$ denotes the quotient of $B$ over~$I$. Finite sums and products of ideals coincide and are ideals; moreover, arbitrary intersections of ideals are ideals. It should also be remarked that, for each left ideal 
$L$ and each strong left ideal $I$ of $B$, we have $LI = L+I$ is a left ideal. Furthermore, for the sake of simplicity, we introduce the following notations (here,~$E$ is a subset of $B$):
\begin{itemize}
    \item The subbrace~$\langle E\rangle$ {\it generated} by~$E$ in $B$ is the smallest subbrace of $B$ containing $E$ with respect to the inclusion. In this case, $E$ is also called a {\it system of generators} for $\langle E\rangle$; and if $E$ is finite, we say that $\langle E\rangle$ is {\it finitely generated}.
    \item The ideal~$E^B$ {\it generated} by~$E$ in $B$ is the smallest ideal of $B$ containing~$E$ with respect to the inclusion. If $E=\{a\}$, we usually denote $\{a\}^B$ simply by $a^B$.
    \item To denote that $C$ is an subbrace of $B$, we write $C\leq B$. To denote that~$I$ is an ideal of $B$, we write $I\trianglelefteq B$.
    \item If $C$ is a maximal subbrace of $B$, we also write $C<\!\!\!\cdot\, B$.
    \item If $C\leq B$, then the maximal ideal~of $B$ contained in $C$ is denoted by~$C_B$.
\end{itemize}

The following are examples of, respectively,  a left ideal and an ideal of a brace that play a central role in the study of nilpotency of braces:\[
\begin{array}{c}
\Fix(B)  = \{a \in B\,|\, \lambda_b(a) = a, \ \text{for every $b\in B$}\}
\end{array}\] and \[\begin{array}{c}
\Soc(B)  = \Ker\lambda \cap \Z(B,+) = \{a\in B\,|\, ab = a+b = b+a, \ \text{for every $b\in B$}\}.
\end{array}\] Also note that $B\ast B$ is always an ideal of $B$.

Minimal and maximal substructures, if they exist, turn out to be essential in every detailed study of an  algebraic structure. Let $S$ be a subbrace (resp. a left ideal or an ideal) of a brace $B$. Then $S$ is a \emph{minimal} subbrace (resp. left ideal or ideal) of $B$ if $S \neq 0$ and $0$ and $S$ are the only subbraces (resp. left ideals or ideals) of $B$ contained in $S$. Moreover, $S$ is a \emph{maximal} subbrace (resp. left ideal or ideal) of $B$ if $S$ is the only proper subbrace (resp. left ideal or  ideal) of $B$ containing~$S$.

\medskip

The following commutator of ideals is introduced in \cite{BournFacchiniPompili23} and plays a key role in the study of nilpotency and solubility in braces.

\begin{definicio}
Let $I,J$ be ideals of a brace $B$. The \emph{commutator} $[I,J]^B$ is defined as
\[ [I,J]^B := \langle [I,J]_+ \cup [I,J]_\cdot \cup \{ij - (i+j):\, i \in I, \, j \in J\}\rangle^B.\]
Clearly,  $[I,J]^B = [J,I]^B$ and $[I,J]^B \leq I \cap J$.
\end{definicio}

A brace $B$ is said to be \emph{abelian} if $[B,B]^B = 0$, i.e. if it is a trivial brace of abelian type. Using this commutator, solubility of braces has been introduced in \cite{BallesterEstebanJimenezPerezC-solublebraces}. If $I$ is an ideal of $B$, the \emph{commutator} or \emph{derived ideal of $I$ with respect to $B$} is defined as $\partial_B(I) = [I,I]^B$. If $B = I$, then we say that $\partial_B(B) = \partial (B)$ is the \emph{commutator} or \emph{derived ideal} of $B$. By repeatedly forming derived ideals, we recursively obtain a descending sequence of ideals
\[ \partial_0(I) = I \geq \partial_1(I) = \partial_B(I) \geq \ldots \geq \partial_n(I) \geq \ldots\]
with $\partial_{n}(I) = \partial_B(\partial_{n-1}(I))$ for every $n\in \mathbb{N}$. We call this series the \emph{derived series of $I$ with respect to $B$}. Clearly,
 $\partial_{n-1}(B)/\partial_{n}(B)$ is an abelian brace for every $n\in \mathbb{N}$, and $B/\partial(B)$ is the greatest abelian quotient in $B$.

\begin{definicio}
We say that an ideal $I$ of $B$ is  \emph{soluble with respect to $B$}, if there exists a non-negative integer $n$ such that $\partial_n(I) = 0$. If $I = B$, we simply say that $B$ is a \emph{soluble brace}, and the smallest non-negative integer~$m$ for which $\partial_m(B)=0$ is the {\it derived length} of $B$.
\end{definicio}

\section{Unifying Universal Algebra definitions of commutators of ideals in braces}\label{sec:unif}

\noindent One of the most fruitful research lines of Universal Algebra is commutator theory. The definition of commutator of ideals we gave in the previous section was inspired by a deep study of the so-called \emph{Huq $=$ Smith condition} for the category of braces (see \cite{BournFacchiniPompili23} for more details). In \cite{BonattoJedlicka22}, a definition of commutator of ideals in braces was introduced by means of another universal algebra point of view (see \cite{freese}). In this section, we show that both definitions coincide, thus answering a question raised in \cite{BonattoJedlicka22}.

\smallskip

We start by briefly recalling the definition of commutator of ideals given in \cite{BonattoJedlicka22}.

\begin{definicio}
Let $B$ be a brace. Given $n\in \mathbb{N}$, an $n$-\emph{polynomial} of $B$ is a map $$p\colon B\times \stackrel{(n)}{\cdots} \times B \rightarrow B$$ whose image $p(x_1,\ldots,x_n)$ is a finite sequence of sums and products of the variables~\hbox{$x_1,\ldots, x_n$} (note that also additive/multiplicative inverses are allowed). $\text{Pol}_n(B)$ denotes the set of all $n$-polynomials of $B$.

A polynomial $p \in \text{Pol}_n(B)$ is said to be \emph{absorbing} if 
\[p(x_1,\ldots,x_{i-1},0,x_{i+1},\dots, x_n) = 0,\]
for every $1\leq i \leq n$ and $x_j\in B$, with $1\leq j\neq i \leq n$.
\end{definicio}

\begin{exemple}\label{exampleabs}
If $B$ is a brace, then the maps $p_1(b_1,b_2) = [b_1,b_2]_+$, $p_2(b_1,b_2) = [b_1,b_2]_\cdot$ and $p_3(b_1,b_2)=b_1b_2-(b_1+b_2)$ are absorbing $2$-polynomials.
\end{exemple}

\begin{definicio} 
Let $I,J$ be ideals of a brace $B$. The commutator $\llbracket I, J \rrbracket^B$ is defined as 
\[  \llbracket I, J\rrbracket^B = \langle \, p(i,j)\,|\, i\in I, j\in J,\, p \in \text{Pol}_2(B) \text{ with $p$ absorbing}\,\rangle^B.\]
\end{definicio}

In \cite{BonattoJedlicka22}, Problem 3.4, the following questions about $\llbracket I, J\rrbracket^B$ were raised.

\begin{problema}\label{quest:commutadors}
Let $B$ be a brace. 
\begin{itemize}
    \item[\textnormal{(1)}] Does the equality $\llbracket I, J\rrbracket^B=\langle I \ast J+ J \ast I + [I,J]_+\rangle^B$ hold?
    \item[\textnormal{(2)}] Does the equality $\llbracket I, J\rrbracket^B = I \ast J  + J \ast I + [I,J]_+$ hold?
\end{itemize}
 
\end{problema} 

Our next theorem gives a positive answer to the first question, thus showing that commutators of ideals defined by means of either the Huq $=$ Smith condition or absorbing polynomials coincide. For the proof, we first remark that the following equality holds in braces.

\begin{nota}\label{eq:commutadors}
Let $B$ be a brace. Then, for every $i,j, x,y,z\in (B,+)$, we have that $i+x+j +y -i + z - j=[-i,-x]_+ + x + [-i,-j]_+ + [-i,-y]^{-j,+}_+ + [-j,-y]_+ + y + [-j,-z]_+ + z$.
\end{nota}

\begin{teorema}
\label{teo:commutador=ast}
Let $I, J$ be ideals of a brace $B$. Then:
\begin{enumerate}
\item[\textnormal{(1)}] $I \ast J + J \ast I + [I,J]_+$ is the least left ideal containing 
\[ X_{I,J} = [I,J]_+ \cup  [I,J]_\cdot \cup \{ij - (i+j):\, i \in I, \, j \in J\}.\]
\item[\textnormal{(2)}] $[I,J]^B = \langle I \ast J + J \ast I + [I,J]_+ \rangle^B$. Thus, $[I,J]^B = I \ast J + J \ast I + [I,J]_+$  if and only if $I \ast J + J \ast I + [I,J]_+$ is an ideal of $B$.
\item[\textnormal{(3)}] $[I,J]^B = \llbracket I, J \rrbracket^B$.
\end{enumerate}
\end{teorema}
\begin{proof}
(1)\quad By \cite[Proposition 1.4]{BonattoJedlicka22}, $I \ast J + J \ast I $ and $[I,J]_+$ are left ideals. Since $[I,J]_+$ is a normal subgroup of $(B,+)$, it follows that $I \ast J + J \ast I + [I,J]_+$ is also a left ideal.

For the inclusion $X_{I,J}\subseteq I \ast J + J \ast I + [I,J]_+$, we need to prove that 
\[ \{ij - (i+j)\,:\, i\in I, \ j \in J\}, \, [I,J]_\cdot \subseteq I \ast J + J \ast I + [I,J]_+\]
Let $i \in I$ and $j \in J$. For the former, observe that $ij - j - i = i + i \ast j - i$. Thus, 
\[ ij - j - i = i + i \ast j - i - i\ast j + i \ast j \in [I,J]_+ + I \ast J.\]
For the latter, using~Eq.~\eqref{eq:prod_ast} repeatedly, it follows that
\begin{align*} iji^{-1}j^{-1} & = iji^{-1} + (iji^{-1})\ast j^{-1} + j^{-1} = \\
& = ij + (ij)\ast i^{-1} + i^{-1} + (iji^{-1})\ast j^{-1} + j^{-1} = \\
& = i + i\ast j + j + (ij)\ast i^{-1} + i^{-1} + (iji^{-1})\ast j^{-1} + j^{-1}.
\end{align*}

\noindent Applying~Eq.~\eqref{eq:identitat_astprod} repeatedly, the above equation becomes
\begin{align*}
&  i + i\ast j + j + i \ast (j \ast i^{-1}) + j \ast i^{-1} + i \ast i^{-1} + i^{-1} +  \\
& +  (ij)\ast (i^{-1}\ast j^{-1}) + i^{-1} \ast j^{-1} + (ij) \ast j^{-1} + j^{-1} = \\[0.2cm]
& = i + i\ast j + j + i \ast (j \ast i^{-1}) + j \ast i^{-1} + i \ast i^{-1} + i^{-1} + \\
& + i \ast\! (j \ast (i^{-1}\ast j^{-1})) + j \ast \! (i^{-1}\ast j^{-1}) + i \ast \!(i^{-1}\ast j^{-1}) + 
i^{-1} \ast j^{-1} + \\
& + i \ast (j \ast j^{-1}) + j \ast j^{-1} + i \ast j^{-1} + j^{-1}
\end{align*}

Observe that $i \ast i^{-1} + i^{-1} = -i + 0 -i^{-1} + i^{-1} = -i$ and $j \ast j^{-1} =  -j - j^{-1}$. Thus, we have
\begin{align}
iji^{-1}j^{-1} & =  i + i\ast j + j + i \ast (j \ast i^{-1}) + j \ast i^{-1} -i + i \ast (j \ast (i^{-1}\ast j^{-1})) \tag{$\star$} \label{eq:sum1}\\
& +  j \ast (i^{-1}\ast j^{-1}) + i \ast (i^{-1}\ast j^{-1}) + i^{-1} \ast j^{-1} + i \ast (j \ast j^{-1}) - j\tag{$\diamond$} \label{eq:sum2}\\
& - j^{-1} + i \ast j^{-1} + j^{-1} \notag
\end{align}
Since $I \ast J, J \ast I \subseteq I \cap J$, Remark \ref{eq:commutadors} applied on~\eqref{eq:sum1}+\eqref{eq:sum2} yields $$iji^{-1}j^{-1}\in I \ast J + J \ast I + [I,J]_+ + (-j^{-1}+i\ast j^{-1} + j^{-1}).$$ But, $-j^{-1} + i \ast j^{-1} + j^{-1} = [j^{-1}, -i\ast j^{-1}]_+ + i\ast j^{-1}\in [I,J]_+ + I \ast J$. Hence, $[I,J]_\cdot \subseteq I\ast J + J \ast I + [I,J]_+$ follows.

Now, let $L$ be the least left ideal of $B$ containing $X_{I,J}$ (note that the arbitrary intersection of left ideals is a left ideal). In order to prove that $I\ast J + J \ast I + [I,J]_+ = L$, we only need to show that $J \ast I\subseteq L$.

Let $j \in J$ and $i \in I$. Then $$j \ast i = [-j\ast i, -j]_+ + (ji - i - j) \in X_{I,J} + (ji-i-j)$$ as $j\ast i \in I \cap J$, so it suffices to prove that $(ji-i-j) \in L$. Since $[j,i]_\cdot \in X_{I,J}\subseteq L$ and $L$ is $\lambda$-invariant, it follows that 
\[ ji = ij[j,i]_\cdot = ij + \lambda_{ij}([j,i]_\cdot) = ij + x\]
for some $x \in L\cap I \cap J$. Therefore,
\begin{align*}
 ji - i - j & = ij + x - ([-j,-i]_+ + i + j) = ij + x - (i+j) + [-i,-j]_+ = \\
 & =  ij - (i+j) + [-(i+j),-x]_+ + x + [-i,-j]_+\in L
\end{align*}
Consequently, $I \ast J + J \ast I + [I,J]_+ = L$.

\medskip

(2)\quad By definition, $[I,J]^B = \langle X_{I,J} \rangle^B$. Then, the previous statement yields $[I,J]^B = \langle I \ast J + J \ast I + [I,J]_+\rangle^B$.

\medskip

(3)\quad By Example \ref{exampleabs}, $[I,J]^B\leq\llbracket I,J\rrbracket^B$. For the other inclusion, let $p$ be an absorbing $2$-polynomial, and choose $i \in I$ and $j\in J$. Recall that $p(i,j)$ is a finite sequence of sums and products of $i$, $j$ and their inverses. Since $[I,J]^B = \langle X_{I,J}\rangle^B$, we can express
\[ p(i,j) + [I,J]^B = (q_i + q_j) + [I,J]^B,\]
where $q_i\in I$ and $q_j\in J$ are such that \[p(i,0) + [I,J]^B = q_i + [I,J]^B\quad\textnormal{and}\quad p(0,j) + [I,J]^B = q_j + [I,J]^B.\] By definition of absorbing polynomial, $p(i,0) = 0 = p(0,j)$, so $q_i+q_j\in [I,J]^B$ and consequently $p(i,j) \in [I,J]^B$.
\end{proof}

The next example gives a negative answer to Problem~\ref{quest:commutadors} (2).

\begin{exemple}\label{ExampleBonatto}
Let \[
\begin{array}{c}
(B,+)  = \langle x, y \,|\, 4x = 4y = 0, \, x + y = y + x \rangle\quad\textnormal{and}\\[0.2cm]
(C,\cdot)  = \langle a, b, c \,|\, c^4=1,\, a^2=b^2 = c^2,\, (ab^{-1})^2 = 1,\, b^{-1}cb = c,\, a^{-1}ca = c\rangle
\end{array}.\] We see that $C$ acts on $B$ via
\[ \arraycolsep=20pt
\begin{array}{lll}
a(x) = x + 2y, & b(x) = x + 2y, & c(x) = 3x + 2y \\
a(y) = -y, & b(y) = 2x + y, & c(y) = 2x + 3y
\end{array}\]
Consider the semidirect product $G$ of $B$ by $C$ with respect to this action $\lambda \colon C\longrightarrow\operatorname{Aut}(B)$. For the sake of clarity, we use multiplicative notation for~$B$ in $G$. Let $D = \langle xa, yb, xyc\rangle \leq G$. It follows that $G$ is a trifactorised group as $BD = DC$, $D\cap C = D\cap B = 1$. By \cite[Lemma 3.2]{BallesterEsteban22}, there exists a bijective $1$-cocycle $\delta\colon C \rightarrow (B,+)$ given by $D = \{\delta(c)c: c \in C\}$ (see Table~\ref{tab-ex_commut-noestrella}). This yields a product in $B$ and we get a brace of abelian type, $(B,+,\cdot)$, which corresponds to \texttt{SmallBrace(16, 73)} in the \textsf{Yang--Baxter} library~\cite{VendraminKonovalov22-YangBaxter-0.10.2} for \textsf{GAP}~\cite{GAP4-12-2}.

\FloatBarrier
\begin{table}[h]
\begin{center}
\begin{tabular}{llllllll}
$g$   & $\delta(g)$ & $g$      & $\delta(g)$ & $g$       & $\delta(g)$ & $g$        & $\delta(g)$ \\ \hline
$1$   & $0$         & $b$      & $y$         & $c$       & $x+y$       & $bc$       & $3x$        \\
$a$   & $x$         & $b^3$    & $2x + 3y$   & $c^3$     & $3x + 3y$   & $bc^{-1}$  & $x + 2y$    \\
$a^2$ & $2x + 2y$   & $ab$     & $x + 3y$    & $ac$      & $2x + y$    & $abc$      & $2y$        \\
$a^3$ & $3x + 2y$   & $ab^{-1}$ & $3x + y$    & $ac^{-1}$ & $3y$        & $abc^{-1}$ & $2x$        \\ \hline
\end{tabular}
\end{center}
\caption{Bijective $1$-cocycle associated with the brace of Example \ref{ExampleBonatto}}
\label{tab-ex_commut-noestrella}
\end{table}
\FloatBarrier

Let $I = \langle 2x, y\rangle\leq (B,+)$. Then $\lambda(I)=I$ and $|I|=8$, so $I$ is an ideal of~$B$. Since $B$ is of abelian type, we compute
\[ I \ast I + [I,I]_+ = I \ast I = \langle u \ast v\,|\, u,v\in I\rangle_+ = \langle 2x \rangle_+ = \{1, 2x\}\]
Therefore, $I \ast I$ is not an ideal of $B$, as $\delta(abc^{-1}) = 2x$ and $\{1,abc^{-1}\}$ is not a normal subgroup of $C$. Hence, $I \ast I \subsetneq [I,I]^B = I$.
\end{exemple}

\section{Central nilpotency of braces}\label{sec:nilp}

\noindent The aim of this section is to develop a standard theory of central nilpotency of braces (for example, canonical lower and upper central series of a brace are studied and a torsion theory is established). In order to have the broadest range of applicability for our results, and also to simplify proofs, we usually deal with concepts that are much more general than central nilpotency (for example, local central-nilpotency, and hypercentrality).

We start by introducing the basic definitions. Central nilpotency of braces was first introduced by using a brace-theoretical analogue of the centre of a group (see~\cite{BonattoJedlicka22} and~\cite{JespersVanAntwerpenVendramin22}). The \emph{centre} of a brace  $B$ (also known as the {\it annihilator ideal} of $B$) is the ideal of $B$ defined as \[ \zeta(B) := \Soc(B) \cap \Fix(B) = \{a \in B \, |\, a + b = b + a = ab = ba, \ \forall\, b \in B\}\] (see \cite{CatinoColazzoStefanelli19}, where it was first introduced in the context of ideal extensions). Thus, abelian braces $B$ are precisely those ones satisfying $\zeta(B) = B$.

In \cite{colazzo} and \cite{Tr23}, the definition of central nilpotency has been extended to include more types of braces (see also \cite{Dixon}, where similar concepts for braces of abelian type are considered). 

\begin{definicio}
Let $B$ be a brace. If $J\leq I$ are ideals of $B$ satisfying\linebreak $I/J\leq\zeta(B/J)$, we say that $I/J$ is a {\it central factor} of $B$.

A {\it $c$-series} of $B$ is a chain $\mathcal{I}$ of ideals  of $B$ such that $0,B\in\mathcal{I}$ and whose factors are central, that is, $I/J\leq\zeta(B/J)$ for all {\it consecutive} elements $J\leq I$ of $\mathcal{I}$ (meaning that there is no $C\in\mathcal{I}$ satisfying $J<C<I$). A {\it complete} $c$-series is just a~\hbox{$c$-series} containing the unions and the intersections of all its members. Since every $c$-series can be completed, we usually consider every $c$-series to be complete. We say that $B$ is {\it $\zeta$-nilpotent} if it admits a $c$-series.

If $B$ has an ascending $c$-series $$0=I_0\leq I_1\leq\ldots I_\alpha\leq I_{\alpha+1}\leq\ldots I_\mu=B$$ (here $\alpha<\mu$ are ordinal numbers), then $B$ is {\it hypercentral}, and the smallest ordinal number $\mu$ for which such an ascending $c$-series exists is the {\it hypercentral length} $n_c(B)$ of $B$. (Note that $I_{\alpha+1}/I_\alpha\leq\zeta(B/I_\alpha)$ for all ordinals $\alpha<\mu$.)

If $B$ has a finite $c$-series \[0 = I_0 \leq I_1\leq \ldots \leq I_n = B,\] then $B$ is {\it centrally nilpotent}; in this case, the smallest non-negative integer for which such a $c$-series exists is referred to as the \emph{class} $n_c(B)$ of $B$. (Note that $I_{i+1}/I_i\leq\zeta(B/I_i)$ for all $0\leq i<n$.)
\end{definicio}

\medskip

Of course, centrally nilpotent braces are hypercentral, and hypercentral braces are $\zeta$-nilpotent, but the converses do not hold. Moreover, subbraces of centrally nilpotent (resp. hypercentral, $\zeta$-nilpotent) braces are centrally nilpotent (resp. hypercentral, $\zeta$-nil\-po\-tent). Also, quotients of centrally nilpotent (resp. hypercentral) braces are still centrally nilpotent (resp. hypercentral), but the consideration of the infinite dihedral group shows that quotients of a $\zeta$-nilpotent brace can be non-$\zeta$-nilpotent. Finally, direct products of hypercentral braces are hypercentral; direct products of finitely many centrally nilpotent braces are centrally nilpotent; and restricted direct products of $\zeta$-nilpotent braces are $\zeta$-nilpotent.

\smallskip

Following \cite{BonattoJedlicka22} and \cite{colazzo}, canonical $c$-series are introduced for a brace~$B$. 

\smallskip

\noindent ($\blacktriangle$) The {\it upper central series} of $B$ is recursively defined by putting: $\zeta_0(B)=0$, $\zeta_{\alpha+1}(B)/\zeta_\alpha(B) = \zeta(B/\zeta_\alpha(B))$ for every ordinal $\alpha$, and $\zeta_\lambda(B)=\bigcup_{\beta<\lambda}\zeta_\beta(B)$ for every limit ordinal $\lambda$. The last term of the upper central series is the {\it hypercentre} $\overline{\zeta}(B)$ of $B$.

\smallskip

\noindent ($\blacktriangledown$) The {\it lower central series} of $B$ is recursively defined by: $\Gamma_1(B)=B$, $\Gamma_{\alpha+1}(B) = \langle \Gamma_\alpha(B) \ast B, B \ast \Gamma_\alpha(B),[\Gamma_\alpha(B), B]_+\rangle_+$ for every ordinal $\alpha$, and $\Gamma_\lambda(B)=\bigcap_{\beta<\lambda}\Gamma_\beta(B)$ for every limit ordinal $\lambda$. Note that since $\Gamma_\alpha(B)$ is an ideal for every $\alpha\leq\mu$, we have that \[ \Gamma_{\alpha+1}(B) = \Gamma_\alpha(B)\ast B  + B \ast \Gamma_\alpha(B) + [\Gamma_\alpha,B]_+ = [\Gamma_\alpha(B),B]^B,\]  for every ordinal $\alpha<\mu$ (see Theorem~\ref{teo:commutador=ast}). The last term of the lower central series is the {\it hypocentre} $\overline{\Gamma}(B)$ of $B$.

\smallskip

The following easily provable all-in-one result shows the relations between the concepts we just introduced (see \cite{Ro72} for the definition of the upper central series $\{Z_\alpha(G)\}_{\alpha\in\operatorname{Ord}(G)}$ of a group $G$).

\begin{proposicio}\label{prop:centre-commut}
Let $B$ be a brace.
\begin{itemize}
    \item[\textnormal{(1)}] \textnormal{(see \cite[Proposition 17]{BallesterEstebanJimenezPerezC-solublebraces})}\; If $J\leq I$ are ideals of $B$, then $I/J \leq \zeta(B/J)$  if and only if $[I,B]^B \leq J$. In particular, if $0 = I_0 \leq I_1 \leq \ldots \leq I_n = B$ is a finite $c$-series, then: 
    \begin{itemize}
        \item[\textnormal{(a)}] $\Gamma_j(B) \leq I_{n-j+1}$ for every $1 \leq j \leq n+1$; so $\Gamma_{n+1}(B) = 0$.
        \item[\textnormal{(b)}] $I_j  \leq \zeta_j(B)$ for every $0 \leq j \leq n$; so $\zeta_n(B) = B$.
        \item[\textnormal{(c)}] $n_c(B)$ is the smallest number $n$ such that $\zeta_n(B)=B$, and the smallest number $n$ such that $\Gamma_{n+1}(B)=0$.
    \end{itemize}
    \item[\textnormal{(2)}] $B$ is hypercentral if and only if $B=\overline{\zeta}(B)$. Moreover, in this case $n_c(B)$ is the smallest ordinal $\alpha$ for which $B=\zeta_\alpha(B)$.
    \item[\textnormal{(3)}] $\zeta_\alpha(B) \subseteq \Z_\alpha(B,+) \cap \Z_\alpha(B,\cdot)$ for all ordinal $\alpha$. In particular, centrally nilpotent  \textnormal(resp. hypercentral\textnormal) braces are braces of nilpotent \textnormal(resp. hypercentral\textnormal) type whose additive group is nilpotent \textnormal(hypercentral\textnormal).
    \item[\textnormal{(4)}] $\partial_{n}(B) \leq \Gamma_{n+1}(B)$ for every $n\in \mathbb{N}$. In particular, centrally nilpotent braces are soluble with derived length less than or equal to their class.
\end{itemize}
\end{proposicio}

\medskip

The following result generalises Gr\"un's lemma for groups (see for instance~\cite{Ro72}, Part~1, p.48).

\begin{teorema}
Let $B$ be a brace such that $\zeta_2(B)>\zeta(B)$. Then either $[B,B]_+$ or $[B,B]_\cdot$ is a proper subset of $B$.
\end{teorema}
\begin{proof}
We may assume by Gr\"un's lemma that $Z(B,+)=Z_2(B,+)$ and $Z(B,\cdot)=Z_2(B,\cdot)$. By Proposition \ref{prop:centre-commut}, $$\zeta_2(B)\subseteq Z_2(B,+)\cap Z_2(B,\cdot)=Z(B,+)\cap Z(B,\cdot).$$ Now, choose $c\in \zeta_2(B)\setminus \zeta(B)$ and consider the map $$\varphi_c:b\in B\mapsto c\ast b\in \zeta(B).$$ Let $b_1,b_2\in B$. Then $$c\ast(b_1+b_2)=c\ast b_1+c\ast b_2$$ because $c\ast b_2\in\zeta(B)$, so $\varphi_c$ is a homomorphism with respect to $+$. Since $c\not\in\zeta(B)$, we have that $c\not\in\operatorname{Ker}(\lambda)$ and consequently $\varphi_c(B)\neq0$. Thus $[B,B]_+$ is properly contained in $B$ and we are done.
\end{proof}

\medskip

In order to see if a brace is centrally nilpotent (resp. hypercentral) or not, we only need to look at its countable subbraces: this is the content of our next result.

\begin{teorema}\label{countablenilp}
Let $B$ be a brace. Then:
\begin{itemize}
\item[\textnormal{(1)}] $B$ is centrally nilpotent of class at most $c$ if and only if its finitely generated subbraces are centrally nilpotent of class at most $c$.
\item[\textnormal{(2)}] $B$ is centrally nilpotent if and only if its countable subbraces are centrally nilpotent.
\item[\textnormal{(3)}] $B$ is hypercentral if and only if its countable subbraces are hypercentral.
\end{itemize}
\end{teorema}
\begin{proof}
For each $u,v\in B$, the symbol $u\circ v$ means that we are performing one (but we do not know which) of the following operations $[u,v]_\cdot,[u,v]_+,u\ast v$. The first statement is a direct consequence of the fact that $\zeta_c(B)$ can be easily characterised as the set of all elements $b\in B$ such that $$(\ldots((b\circ b_1)\circ\ldots)\circ b_c)=0$$ for all $b_1,\ldots,b_c\in B$. 

\medskip

\noindent(2)\quad Only one implication is in doubt. Thus, assume all countable subbraces of $B$ are centrally nilpotent but $B$ is not centrally nilpotent. By (1), for each $c\in\mathbb{N}$, there is a finitely generated centrally nilpotent brace $B_c$ whose class is at least $c$. Then $C=\langle B_i\,:\, i\in\mathbb{N}\rangle$ is a countable subbrace of $B$ that is not centrally nilpotent, a contradiction.

\medskip

\noindent(3)\quad Only one implication is in doubt. Thus, assume all countable subbraces of $B$ are hypercentral, and that $B$ is not hypercentral. We may certainly assume that $\zeta(B)=0$. Let $0\neq x\in B$. Then there are sequences of non-zero elements $$a_1,a_2,\ldots,a_n,\ldots\quad\textnormal{and}\quad x=b_1,b_2\ldots,b_n,\ldots$$ of $B$ such that $b_{n+1}=b_n\circ a_n$ for all $n\in\mathbb{N}$. Let $C=\langle a_i,b_j\,:\, i,j\in\mathbb{N}\rangle$. Thus,~$C$ is countable and so hypercentral. Now, for each $i\in\mathbb{N}$, let $\alpha_i$ be the smallest ordinal $\beta$ such that $b_i\in\zeta_\beta(C)$. It follows that $$\alpha_1>\alpha_2>\ldots>\alpha_n>\ldots$$ is a strictly descending chain of ordinal numbers, a contradiction.\end{proof}

\medskip

In studying the structure of an arbitrary group, local analogs of nilpotence are very useful. Similarly, the following definition is crucial for us in studying centrally nilpotent braces (see \cite{colazzo}).

\begin{definicio}
A brace $B$ is {\it locally centrally-nilpotent} if every finitely generated subbrace is centrally nilpotent.
\end{definicio}

\medskip

Of course, every subbrace/quotient of a locally centrally-nilpotent brace is still locally centrally-nilpotent, and also restricted direct products of locally centrally-nilpotent braces are locally centrally-nilpotent. Moreover, by~\cite{Tr23},~Co\-rol\-la\-ry~3.6, every hypercentral brace is locally centrally-nilpotent, but the converse does not hold. As a consequence of the following result, we see that every locally centrally-nilpotent brace is $\zeta$-nilpotent.

\begin{teorema}\label{locallnilp}
Let $B$ be a locally centrally-nilpotent brace.
\begin{itemize}
    \item[\textnormal{(1)}] If $I$ is any minimal ideal of $B$, then $I\leq\zeta(B)$.
    \item[\textnormal{(2)}] If $M$ is any maximal subbrace of $B$, then $M$ is an ideal of $B$.
\end{itemize}

\noindent In particular, every minimal ideal of $B$ has prime order, and $\partial(B)$ is contained in every maximal subbrace of $B$.
\end{teorema}
\begin{proof}
(1)\quad Suppose that $I\not\leq\zeta(B)$. Then there exist elements $b\in B$ and $x\in I$ such that $\mathcal{S}=\{[b,x]_+,[b,x]_\cdot,x\ast b\}\neq\{0\}$. Let $c\in\mathcal{S}\setminus\{0\}$. Since~$I$ is a minimal ideal of~$B$, the ideal generated by $c$ in $B$ is $I$, so there are elements $y_1,\ldots,y_n\in B$ such that~$x$ belongs to the ideal generated by $c$ in $S=\langle b,c,y_1,\ldots,y_n\rangle$. 

Let $J=x^S$. Since $S$ is centrally nilpotent, there is a finite chain $$0=J_0< J_1<\ldots< J_m=J$$ of ideals of $S$ such that $J_{i+1}/J_i\leq\zeta(S/J_i)$. Choose $\ell\in\mathbb{N}$ with $c\in J_\ell\setminus J_{\ell-1}$; clearly, $\ell\neq0,m$ because $c$ is a non-zero element of one of the following types: $[b,x]_+,[b,x]_\cdot,x\ast b$.  Now, $$x^S\leq c^S\leq c^S+J_{\ell-1}=\langle c\rangle+J_{\ell-1}\leq J_\ell<J_m=J=x^S,$$ a contradiction.

\medskip

\noindent(2)\quad Assume $M$ is not an ideal. If $B\ast B\leq M$, then $(M,+)/(B\ast B,+)$ is a maximal subgroup of the locally nilpotent group $(B,+)/(B\ast B,+)$, so it is even normal, and it follows that $M$ is an ideal of $B$, a contradiction. Thus, there exists an element $x\in B\ast B\setminus M$. Since $M$ is a maximal subbrace of~$B$, we have that $B=\langle M,x\rangle$. Then there is a finite subset $L$ of $B$ such that $x\in \langle L\rangle\ast\langle L\rangle$. For each $y\in L$, let $B_y$ be a finite subset of $M$ such that $y\in\langle B_y\cup\{x\}\rangle$. Put $$D=\langle B_y\,:\, y\in L\rangle\quad\textnormal{and}\quad E=\langle D,x\rangle,$$ so $E$ is finitely generated and $L\subseteq E$. Now, $E$ is centrally nilpotent, and \hbox{$x\not\in D\subseteq M$.} Let $N$ be a subbrace of $E$ which is maximal with respect to containing $D$ but not~$x$. Since $E=\langle D,x\rangle$, we see that $N$ is actually a maximal subbrace of $E$. Since $E$ is centrally nilpotent, there is $n\in\mathbb{N}$ such that $\zeta_n(E)\leq N$ but $\zeta_{n+1}(E)\not\leq N$. Then $N$ is a proper ideal of $\zeta_{n+1}(E)+N$ and so $N\trianglelefteq E$, since $N$ is a maximal subbrace of $E$. Therefore $E/N$ is centrally nilpotent, and so $x\in E\ast E\leq N$, a contradiction.
\end{proof}

\medskip

In \cite{BallesterEstebanPerezC-JH}, chief factors of braces are introduced and shown to play a key role in its ideal structure. Let $I$ and $J$ be ideals of a brace $B$ such that $J \leq I$. The quotient $I/J$ is said to be a \emph{chief factor} of $B$ if $I/J$ is a minimal ideal of~$B/J$. A chain $\mathcal{C}$ of ideals of $B$ is a \emph{chief series} of $B$ if $0,B\in\mathcal{C}$ and $I/J$ is a chief factor of~$B$ whenever $J<I$ are consecutive terms of $\mathcal{C}$. By Zorn's lemma, every brace has a (possibly infinite) chief series. In \cite{BallesterEstebanPerezC-JH}, a brace $B$ is proved to have a finite chief series if and only if it is {\it noetherian} (that is, every ascending chain of ideals is eventually stationary) and {\it artinian} (that is, every descending chain of ideals is eventually stationary). By The\-o\-rem~\ref{locallnilp}, every chief series of a locally centrally-nilpotent brace is a $c$-series, so we have the following result.

\begin{corolari}\label{localezetanilp}
Let $B$ be a locally centrally-nilpotent brace. Then $B$ is~\hbox{$\zeta$-nil}\-po\-tent.
\end{corolari}

\begin{nota}
{\rm It should be noted that the proof of Theorem \ref{locallnilp} (1) proves much more than we stated. In fact, let $\mathfrak{Z}$ be the class of all braces in which every chief factor is central. Moreover, let {\bf L}$\mathfrak{Z}$ be the class of all braces in which every finite subset $F$ is contained in a subbrace $C_F\in\mathfrak{Z}$. The proof of~The\-o\-rem~\ref{locallnilp} (1) can be modified to show that {\bf L}$\mathfrak{Z}=\mathfrak{Z}$.

More in detail, using the notation of the first half of the proof of The\-o\-rem~\ref{locallnilp}~(1), we get that $S$ is contained in a subbrace $T\in\mathfrak{Z}$. Let $J=x^T$.

Since $[J,T]^T$ is an ideal of $T$ containing $c$, we have that $x\in c^T\leq[J,T]^T$ and so $J=x^T=[J,T]^T$. Finally, let $M$ be a maximal ideal of $T$ contained in~$J$ and such that $x\not\in M$. Then $J/M$ is a chief factor of $T$ and so~\hbox{$[J,T]^T\leq M$,} a contradiction.}
\end{nota}

\medskip

It follows from Corollary \ref{localezetanilp} that any non-zero ideal of a locally centrally-nilpotent brace contains a (non-zero) central factor of the whole brace. In case of a hypercentral brace we can say more.

\begin{lema}\label{dontworry}
Let $B$ be a hypercentral brace. If $I$ is any non-zero ideal of~$B$, then $I\cap\zeta(B)\neq0$.
\end{lema}
\begin{proof}
Let $\alpha$ be the smallest ordinal number such that $J=I\cap\zeta_\alpha(B)\neq0$. Then $\alpha$ is successor and $I\cap \zeta_{\alpha-1}(B)=0$. Now, $[J,B]^B\leq J\cap\zeta_{\alpha-1}(B)=0$ and so $J\leq\zeta(B)$.
\end{proof}

\begin{corolari}\label{cor:caract_nil-chieffact}
Let $B$ be a brace having a finite chief series \textnormal(resp. an ascending chief series\textnormal) $\mathcal{I}$. Then $B$ is centrally nilpotent \textnormal(resp. hypercentral\textnormal) if and only if every chief factor of $\mathcal{I}$ is central.
\end{corolari}

Our next two subsections deal with the torsion theory of locally centrally-nilpotent braces and with the problem of defining a suitable index for subbraces. Before delving into them, we note that some important results for nilpotent groups do not hold for braces. 

\begin{itemize}
    \item Bearing in mind the normaliser condition for nilpotent groups, the idealiser of a subbrace is introduced in \cite{JespersVanAntwerpenVendramin22}: given a subbrace $S$ of a brace~$B$, the idealiser of $S$ is defined as the largest subbrace $N$ of~$B$ such that $S$ is an ideal of $N$. It is then stated that every subideal is properly contained in its idealiser (see Section \ref{sec:index} for the definition of subideal).~Example~\ref{ex:noidealiser} in Section~\ref{wex} shows that the idealiser of a subbrace does not exist in general. We note however that if $C$ is a subbrace of a brace $B$, then one can define the largest strong left ideal $N_B(C)$ of $B$ additively and multiplicatively normalising $B$ and such that $\lambda_x(C)=C$ for every~\hbox{$x\in N_B(C)$} --- but such a strong left ideal need not to contain~$C$.
    \item Example \ref{ex:nofitting} shows that there is no analogue of Fitting theorem for centrally nilpotent ideals.
    \item Example \ref{ex:triv_nosubnilid} shows that an abelian subideal need not be contained in a centrally nilpotent ideal.
    \item The ideal structure of the brace listed as \texttt{SmallBrace(32, 24003)} in the \textsf{YangBaxter} library \cite{VendraminKonovalov22-YangBaxter-0.10.2} for \textsf{GAP} \cite{GAP4-12-2} is described in \cite{BallesterEstebanJimenezPerezC-solublebraces}. This brace~$B$ has only a unique maximal subbrace $I$, which is also its only non-zero proper ideal. Moreover, $\partial(B) = I$. Nevertheless, $B$ is not centrally nilpotent as it is not even soluble. This shows that a finite brace whose maximal subbraces are ideals need not be soluble. The same example shows that a finite brace whose subbraces are subideals need not be centrally nilpotent (see Example \ref{ex:D} for more details), although an easy induction shows that they are at least weakly soluble in the sense explained in~\cite{BallesterEstebanJimenezPerezC-solublebraces}.
\end{itemize}

\subsection{Torsion theory}\label{sec:tor}

\noindent The aim of this subsection is to establish a torsion theory for locally centrally-nilpotent braces. We start with some definitions.

\begin{definicio}
Let $B$ be a brace. The subset of all periodic elements of~$(B,+)$ is denoted by $T_+(B)$, while that of all periodic elements of $(B,\cdot)$ is denoted by $T_\cdot(B)$. 
Moreover:
\begin{itemize}
    \item An element $b$ of $B$ is {\it periodic} if $\langle b\rangle$ is finite. The {\it order} of $b$ is $|\langle b\rangle|$. If~$\pi$ is any set of prime numbers, then  $b$ is a {\it $\pi$-element} if its order is a~\hbox{$\pi$-number.} A {\it $\pi$-subbrace} is just a subbrace containing only $\pi$-elements. Finally, $B$ is {\it periodic} if every element of $B$ is periodic.
    \item $B$ is {\it torsion-free} if every element $b\in B$ is either zero or is non-periodic.
    \item $B$ is {\it locally finite} if every finitely generated subbrace of $B$ is finite.
    \item $B$ has {\it finite exponent} $n$ if $B$ is periodic and $n$ is the smallest positive integer such that $b^n=nb=0$ for all $b\in B$.
\end{itemize}
\end{definicio}

\medskip

Clearly, every locally finite brace is periodic but the converse does not hold. The following result shows that in the context of locally centrally-nilpotent braces we can precisely identify the set of all periodic elements of~$B$. 

\begin{teorema}
Let $B$ be a locally centrally-nilpotent brace. Then:
\begin{itemize}
    \item[\textnormal{(1)}] $T_+(B)=T_\cdot(B)$.
    \item[\textnormal{(2)}] $T_+(B)$ is an ideal of $B$.
    \item[\textnormal{(3)}] $T_+(B/T_+(B))=0$.
    \item[\textnormal{(4)}] If $B$ is periodic, then $B$ is locally finite.
    \item[\textnormal{(5)}] $B$ is locally finite\!$\iff$\! $(B,+)$ is locally \hbox{finite\!$\iff$\!$(B,\cdot)$ is locally finite.}
\end{itemize}
\end{teorema}
\begin{proof}
The proof of (1)--(3) is an easy consequence of \cite{JespersVanAntwerpenVendramin22}, Proposition 4.2. Let us prove (4). Assume $B$ is periodic and finitely generated. Then $(B,+)$ is a periodic nilpotent group. Moreover, by Theorem 3.7 of \cite{Tr23}, $(B,+)$ is also finitely generated. Thus $(B,+)$ is finite and (4) is proved. Finally, (5) is an obvious consequence of \cite{Tr23}, Theorem 3.7.
\end{proof}

\medskip

Let $B$ be a brace, and let $p$ be a prime. The {\it Sylow $p$-subbrace} of $B$ is just a maximal element of the set of all its $p$-subbraces with respect to the inclusion. Our next result shows that the Sylow subbraces of a locally centrally-nilpotent brace are ideals and that they coincide with the additive/multiplicative Sylow subgroups.

\begin{teorema}
\label{prop:car_sylideals}
Let $B$ be a locally finite brace. Then, $B$ is locally centrally-nilpotent if and only if, for every prime $p$, $\Syl_p(B,+) = \Syl_p(B,\cdot) = \{B_p\}$, $B_p$ is locally centrally-nilpotent and $B$ is the direct product of the $B_p$'s. 
\end{teorema}
\begin{proof}
Only one direction is in doubt. Since $B$ is locally finite, we may assume~$B$ is finite and centrally nilpotent. Let $p$ be a prime. Since both $(B,+)$ and $(B, \cdot)$ are nilpotent groups, there exist Sylow $p$-subgroups $B_p \unlhd (B,+)$ and $\bar{B}_p \unlhd (B,\cdot)$. Observe that $B_p$ is also $\lambda$-invariant, as it is a characteristic subgroup of $(B,+)$. Therefore, $B_p = \bar{B}_p$ is an ideal of $B$.
\end{proof}

With respect to the definition of order of an element, we note the following interesting fact.

\begin{proposicio}
Let $B$ be a brace whose additive and multiplicative groups are cyclic. Then there is $x\in B$ which is a generator of both $(B,+)$ and $(B,\cdot)$.
\end{proposicio}
\begin{proof}
By Theorem 4.6 of \cite{stefan}, we may assume $B$ is finite. If $\operatorname{Ker}(\lambda)=0$, then $(B,\cdot)$ embeds into $\operatorname{Aut}(B,+)$, a contradiction. Thus, $\zeta(B)\neq0$. Iterating this argument, we see that $B$ is centrally nilpotent, so $B$ factorizes into the direct product of its Sylow $p$-subgroups. It is therefore possible to assume that $B$ has prime power order $p^n$.

Let $I$ be a subbrace of $\zeta(B)$ of order $p$. By induction there is an element $x\in B$ such that $x+I$ is a both a generator of $(B/I,+)$ and $(B/I,\cdot)$. If \hbox{$\langle x^{p^{n-1}}\rangle_\cdot\cap I=0$,} then  $(B,\cdot)$ is not cyclic, a contradiction. Thus, $\langle x^{p^{n-1}}\rangle_\cdot=I$ and $x$ is a generator of $(B,\cdot)$. Similarly, $x$ is a generator of $(B,+)$.
\end{proof}

\medskip

Our next result is a huge generalisation of \cite{JespersVanAntwerpenVendramin22}, Lem\-ma~4.1. In order to state it, we need the following definition.

\begin{definicio}\label{free}
{\rm Let $B$ be a brace, and let $\pi$ be a set of prime numbers. We say that $B$ is {\it $\pi$-free} if it does not contain $\pi$-elements.}
\end{definicio}

\medskip

Obviously, a trivial brace $B$ is $\pi$-free if and only if  $(B,+)$ and/or $(B,\cdot)$ are \hbox{$\pi$-free} as groups.

\begin{teorema}\label{moltobello}
Let $B$ be a brace and let $\pi$ be a set of primes. If~$\zeta(B)$ is~\hbox{$\pi$-free,} then each factor of the upper central series of $B$, and therefore the hypercentre of $B$, is $\pi$-free.
\end{teorema}
\begin{proof}
Suppose the theorem is false and let $\alpha$ be the first ordinal such that $\zeta_{\alpha+1}(B)/\zeta_\alpha(B)$ is not $\pi$-free; in particular, there is $x\in \zeta_{\alpha+1}(B)\setminus\zeta_\alpha(B)$ such that $x^m\in\zeta_\alpha(B)$ for some positive $\pi$-number $m$. We divide the proof in two parts according to $\alpha$ being limit or not.

Suppose first $\alpha$ is limit. Then $x^m\in\zeta_{\beta+1}(B)\setminus\zeta_\beta(B)$ for some $\beta<\alpha$. Since $x\notin\zeta_\alpha(B)$, there is $b\in B$ and $\gamma\geq\beta$ such that one of the elements $x\ast b$, $[x,b]_+$, $[x,b]_\cdot$ belongs to $\zeta_{\gamma+1}(B)\setminus\zeta_\gamma(B)$; call $c$ such an element. Assume $c=x\ast b$. Then $$(x\ast b)^m\equiv x^m\ast b\quad\textnormal{(mod $\zeta_{\gamma}(B)$)}$$ and so $x^m\ast b\in\zeta_\beta(B)\leq\zeta_\gamma(B)$. Therefore $(x\ast b)^m\in\zeta_\gamma(B)$. But $x\ast b\in\zeta_{\gamma+1}(B)$ and $\zeta_{\gamma+1}(B)/\zeta_\gamma(B)$ is $\pi$-free, so $c = x\ast b\in\zeta_\gamma(B)$, a contradiction.
Similarly, we deal with the cases in which $c=[b,x]_\cdot$ and $c=[b,x]_+$.

Suppose now that $\alpha$ is successor; in this case, we may assume $\alpha=1$, so $x\in\zeta_2(B)\setminus\zeta(B)$, $\zeta(B)$ is $\pi$-free, and $x^m\in\zeta(B)$. Put $C=\langle x\rangle+\zeta(B)$. By~The\-o\-rem~3.5 of \cite{colazzo}, we have that $|C\ast C|$ is a $\pi$-number. On the other hand, $C\ast C\leq\zeta(B)$, and so $C\ast C=0$. Thus, $x^m=mx$.

Let $b$ be any element of $B$. Then $m[x,b]_+=[mx,b]_+=0$, so $[x,b]_+=0$. Similarly, $[x,b]_\cdot=x\ast b=0$. Therefore $x$ belongs to $\zeta(B)$, the final contradiction.
\end{proof}

\begin{corolari}
Let $B$ be a brace. If $\zeta(B)$ is torsion-free, then $\zeta_{\alpha+1}(B)/\zeta_\alpha(B)$ is torsion-free for every ordinal $\alpha$.
\end{corolari}

\medskip

Conversely, if we have information on the exponent of $\zeta(B)$, then we can obtain information on the exponent of the factors of the upper central series.

\begin{teorema}\label{dixmier}
Let $B$ be a brace. If $\zeta(B)$ has exponent $n$, then $\zeta_{i+1}(B)/\zeta_i(B)$ has exponent dividing $n^{2^i}$ for each positive integer $i$.
\end{teorema}
\begin{proof}
It is enough to show that $\zeta_2(B)/\zeta(B)$ has exponent dividing $n^2$. Let $b\in\zeta_2(B)$ and $a\in B$. Then $b\ast a\in\zeta(B)$, so $$b^n\ast a=n(b\ast a)=0\quad\textnormal{and}\quad [a,b^n]_\cdot=[a,b]^{n}_\cdot=0.$$ Thus, if we put $c=nb^n=b^{n^2}$, then $c\in\operatorname{Ker}(\lambda)\cap Z(B,\cdot)$. But also $$[a,nb^n]_+=n[a,b^n]_+=0$$ and so $c\in \zeta(B)$. 
\end{proof}

\begin{corolari}
Let $B$ be a brace such that $\zeta(B)$ has exponent $n$. If $B$ is centrally nilpotent of class $c$, then $B$ has exponent at most $n^{2^c-1}$.
\end{corolari}

\subsection{The index of a subbrace}\label{sec:index}

\noindent One of the ways in which we cope with infinite groups is by using finite-index subgroups. For braces, things are much more complicated, since we could deal at the same time with two distinct unrelated indices. The aim of this section is to study the ‘‘index’’ of a subbrace with particular emphasis to case of locally centrally-nilpotent braces. The following definition provides us with an invaluable tool in this study.

\begin{definicio}
Let $C$ be a subbrace of the brace $B$. We say that $C$ is {\it serial} in $B$ if there is a chain of subbraces $\mathcal{C}$ connecting $C$ to $B$ such that if~\hbox{$D<E$} are consecutive elements of $\mathcal{C}$, then $D\trianglelefteq E$ --- as in the case of $c$-series, we usually assume that these chains of subbraces are {\it complete}, meaning that they contain arbitrary unions and intersections of their members.

Now, $C$ is {\it ascendant} (resp. {\it descendant}) if $\mathcal{C}$ can be well-ordered (resp. inversely well-ordered) with respect to the inclusion and its order type is $\lambda$ (resp. the inverse of $\lambda$) for some ordinal number $\lambda$. If $C$ is ascendant, then~$\mathcal{C}$ can be written as \[
\begin{array}{c}\label{asc}
C=C_0\trianglelefteq C_1\trianglelefteq\ldots C_\alpha\trianglelefteq C_{\alpha+1}\trianglelefteq\ldots C_\lambda=B\tag{$\triangle$},
\end{array}
\] where $\alpha<\lambda$ are ordinal numbers; while, if $C$ is descendant, then $\mathcal{C}$ takes the form \[
\begin{array}{c}\label{desc}
C=C_\lambda\ldots\trianglelefteq C_{\alpha+1}\trianglelefteq C_{\alpha}\trianglelefteq\ldots\trianglelefteq C_1\trianglelefteq C_0=B,\tag{$\square$}
\end{array}
\] where $\alpha<\lambda$ are ordinal numbers. If $\mathcal{I}$ is finite, we say that $C$ is {\it subideal}.

If $C$ is ascendant (resp. descendant) in $B$, then the smallest ordinal number $\lambda$ for which there is a chain of subbraces of type \eqref{asc} (resp. of type \eqref{desc}) is the {\it ascendant length} (resp. {\it descendant length}) of $C$ in $B$. In case $C$ is subideal, then the ascendant length of $C$ in $B$ is finite and is also called the {\it subideal defect} of $C$ in $B$. 
\end{definicio}

Let $C$ be a subbrace of a brace $B$. Put $C^{B,0}:=C$ and recursively define $C^{B,\alpha+1}=C^{C^{B,\alpha}}$ for every ordinal $\alpha$, and $C^{B,\lambda}=\bigcap_{\alpha<\lambda}C^{B,\alpha}$ for every limit ordinal $\lambda$. The family $\{C^{B,\alpha}\}_{\alpha\in\operatorname{Ord}}$ is the {\it ideal closure series} of $C$ in $B$. It is easy to show that $C$ is descendant (resp. subideal) in $B$ if and only if $C=C^{B,\mu}$ for some ordinal $\mu$ (resp. for some finite ordinal $\mu$). If $C$ is descendant (resp. subideal), then the descendant length (resp. the subideal defect) of $C$ is then the smallest ordinal number $\lambda$ for which $C=C^{B,\lambda}$.

\medskip

The following easy lemma contains all the basic properties of subideal, ascendant, descendant and serial subbraces. 

\begin{lema}
Let $B$ be a brace. 

\begin{itemize} 
    \item Every subideal of $B$ is ascendant, descendant and serial.
    \item Ascendant \textnormal(resp. descendant\textnormal) subbraces are serial.
    \item If $C$ is subideal \textnormal(resp. ascendant, descendant, serial\textnormal) in $B$, and $D\leq B$, then $C\cap D$ is subideal \textnormal(resp. ascendant, descendant, serial\textnormal) in $D$.
    \item If $C$ is subideal \textnormal(resp. ascendant\textnormal), then $CI/I$ is subideal \textnormal(resp. ascendant\textnormal) in $B/I$ for every ideal $I$ of $B$.
    \item If $C$ is subideal in $B$ of defect $n$, then $C$ is subideal in $C^B$ of defect $n-1$.
\end{itemize}
\end{lema}

\medskip

Serial (resp. ascendant) subbraces play a relevant role in the theory of locally centrally-nilpotent braces (resp. hypercentral braces), as the following result shows.

\begin{lema}\label{annihilatorascendant}
Let $B$ be a brace.

\begin{itemize}
    \item[\textnormal{(1)}] If $B$ is hypercentral, then every subbrace $C$ of $B$ is ascendant.
    \item[\textnormal{(2)}] If $B$ is centrally-nilpotent, then every subbrace $C$ of $B$ is subideal.
    \item[\textnormal{(3)}] If $B$ is locally centrally-nilpotent, then every subbrace $C$ of $B$ is serial.
\end{itemize}
\end{lema}
\begin{proof}
(1)\quad Let $\lambda$ be the hypercentral length of $B$. Since $C$ is an ideal of $C+\zeta(B)$, we see that $$C\trianglelefteq C+\zeta(B)\trianglelefteq \ldots C+\zeta_\alpha(B)\trianglelefteq C+\zeta_{\alpha+1}(B)\trianglelefteq\ldots C+\zeta_{\lambda}(B)=B$$ is an ascending chain of subbraces of $B$ connecting $C$ to $B$. 

\medskip

\noindent(2)\quad The proof is the same of (1).

\medskip

\noindent(3)\quad Zorn's lemma implies that there is a maximal chain of subbraces between~$C$ and $B$. By Lemma \ref{locallnilp} (2), if $D<E$ are consecutive terms of this chain, then $D$ is an ideal of $E$. Therefore $C$ is serial in $B$.
\end{proof}

\begin{nota}\label{potenziare}
Let $B$ be a brace and let $C$ be a (strong) left ideal of $B$. The proof of Lemma \ref{annihilatorascendant} can actually be employed to prove that if we have an ascending $c$-series of $B$, then there is an ascending chain of (strong) left ideals connecting $C$ to $B$.
\end{nota}

\medskip

Lemma \ref{annihilatorascendant} has some rather interesting consequences concerning the ‘‘index’’ of a subbrace.

\begin{definicio}
{\rm Let $B$ be a brace. A subbrace $C$ of $B$ is said to have {\it finite index} in $B$ if both $n_+=|(B,+):(C,+)|$ and $n_\cdot=|(B,\cdot):(C,\cdot)|$ are finite; if~\hbox{$n_+=n_\cdot=n$,} we define the {\it index $|B:C|$} of $C$ in $B$ as $n$. If $C$ has not finite index, we say that $C$ has {\it infinite index}.}
\end{definicio}

\medskip

The following easy result comprises some of the basic statements about subbraces of finite index.

\begin{lema}\label{basicpropertiesindex}
Let $B$ be a brace, $C,D\leq B$ and $I\trianglelefteq B$. Then:
\begin{itemize}
    \item[\textnormal{(1)}] If $C$ and $D$ have finite index in $B$, then $C\cap D$ has finite index in $B$.
    \item[\textnormal{(2)}] Suppose $C\leq D$. If $C$ has finite index in $D$, and $D$ has finite index in~$B$, then $C$ has finite index in $B$. Moreover, if $|D:C|$ and $|B:D|$ exist, then also $|B:C|=|B:D|\cdot|D:C|$ exists.
    \item[\textnormal{(3)}] If $I$ has finite index, then $|B:I|$ exists and is equal to $|B/I|$.
\end{itemize}
\end{lema}

\medskip

The following result shows that serial subbraces of finite index have a well-defined index.

\begin{lema}\label{ascendantindex}
Let $C$ be a serial subbrace of the brace $B$. The following conditions are equivalent:
\begin{itemize}
    \item[\textnormal{(1)}] $|(B,+):(C,+)|$ is finite.
    \item[\textnormal{(2)}] $|(B,\cdot):(C,\cdot)|$ is finite.
    \item[\textnormal{(3)}] $C$ has finite index in $B$.
    \item[\textnormal{(4)}] $|B:C|$ exists.
\end{itemize}

\noindent In particular, if any of the above equivalent statements hold, then all the indices are equal.
\end{lema}
\begin{proof}
Clearly, (4) $\implies$ (1), (2) and (3). Assume (1). Since $C$ is serial in~$B$, there is a chain $\mathcal{C}$ of subbraces connecting $C$ to $B$, and in which \hbox{$E\trianglelefteq F$} whenever $E<\!\!\!\cdot\,F$. Looking at the corresponding additive parts of the members of $\mathcal{C}$, we see that $\mathcal{C}$ is actually finite, so $C$ is subideal in $B$. We prove the result by induction on the subideal defect $n$ of $C$ in $B$. If $n\leq1$, then $C$ is an ideal of $B$ such that $(B,+)/(C,+)$ is finite, so $B/C$ is finite and we are done. Assume $n>1$ and let $D=C^B$. Then the subideal defect of $C$ in $D$ is strictly less than $n$ and so induction yields that $|D:C|$ exists. Since $|B:D|$ trivially exists, we have that $|B:C|$ exists by~Lem\-ma~\ref{basicpropertiesindex}. Thus, (4) is proved. Similarly, we can prove that (2) implies (4), and that (3) implies~(4).
\end{proof}

\medskip

A combination of Lemma \ref{ascendantindex} and Lemma \ref{annihilatorascendant} shows that every finite-index subbrace of a locally centrally-nilpotent brace has a well-defined index. Our next result is a considerable extension of this fact.

\begin{teorema}\label{indexexists}
Let $B$ be a brace having an ascendant chain of ideals $$0=I_0\leq I_1\leq\ldots I_\alpha\leq I_{\alpha+1}\leq\ldots I_\lambda=B$$ such that $I_{\beta+1}/I_\beta$ is either finite or locally centrally-nilpotent for all ordinal numbers $\beta<\lambda$. If $C$ is a subbrace of $B$, then the following are equivalent:
\begin{itemize}
    \item[\textnormal{(1)}] $|(B,+):(C,+)|$ is finite.
    \item[\textnormal{(2)}] $|(B,\cdot):(C,\cdot)|$ is finite.
    \item[\textnormal{(3)}] $C$ has finite index in $B$.
    \item[\textnormal{(4)}] $|B:C|$ exists.
\end{itemize}
\end{teorema}
\begin{proof}
We prove the result by induction on $\lambda$. To this aim it is sufficient to show that (1) implies (4).

If $\lambda\leq1$, then $B$ is either finite or locally centrally-nilpotent. The former case is obvious, while the latter is a consequence of~Lem\-ma~\ref{annihilatorascendant}~(3) and~Lem\-ma~\ref{ascendantindex}. Assume $\lambda>1$.

Suppose $\lambda$ is successor. Since $(C,+)$ has finite index in $(B,+)$, it follows that $(C\cap I_{\lambda-1},+)$ has finite index $n$ in $(I_{\lambda-1},+)$. By induction, $$n=\big|(I_{\lambda-1},\cdot):(C\cap I_{\lambda-1},\cdot)\big|=\big|(CI_{\lambda-1},\cdot):(C,\cdot)\big|=\big|(C+I_{\lambda-1},+):(C,+)\big|.$$ Thus the index $|C+I_{\lambda-1}:C|$ exists. Since also the index $|B:C+I_{\lambda-1}|$ exists, it follows that the index $|B:C|$ exists.

Now, assume $\lambda$ is limit. Let $F_+$ be a transversal for $(C,+)$ in $(B,+)$; in particular, $F$ is finite. Also, let $F_\cdot$ be a transversal for $(C,\cdot)$ in $(B,\cdot)$. Suppose $|F_\cdot|>|F_+|$, and let $E_\cdot$ be a finite subset of $F_\cdot$ such that $|E_\cdot|>|F_+|$. Then there is an ordinal number $\mu<\lambda$ such that $F_+\cup E_\cdot\subseteq I_\mu$. By induction, $$\big|(B,+):(C,+)\big|=\big|(I_\mu,+):(C\cap I_\mu,+)\big|=\big|(I_\mu,\cdot):(C\cap I_\mu,\cdot)\big|,$$ a contradiction. Thus $|F_\cdot|\leq |F_+|$. By a symmetric argument, \hbox{$|F_+|\leq |F_\cdot|$} and hence the index $|B:C|$ exists.
\end{proof}

\medskip

The range of applicability of Theorem \ref{indexexists} is not restricted to local centrally-nilpotent brace. It follows in fact from \cite{colazzo}, Theorems 3.14, that Theorem \ref{indexexists} applies even to any {\it good} brace with {\it property} (S) (see \cite{colazzo} for the definitions).

\medskip

We end this discussion by showing that subbraces of finite index can sometimes be employed to prove the existence of large proper ideals.

\begin{teorema}
Let $B$ be a brace such that $B/\zeta_2(B)$ is finite. If $C$ is any finite-index subbrace of $B$, then $B/C_B$ is finite.
\end{teorema}
\begin{proof}
Without loss of generality we may assume $C_B=0$, so in particular, $C\cap\zeta(B)=0$ and $\zeta(B)$ is finite. Moreover, we may replace $C$ by $C\cap \zeta_2(B)$, so assuming $C\leq\zeta_2(B)$. Then $C+\zeta(B)$ is an ideal of $B$.

Since $C\cong C+\zeta(B)/\zeta(B)$, we have that $C$ is an abelian brace. Let $n=|\zeta(B)|$. Then Theorem \ref{dixmier} shows that $$C^{n^2,\cdot}\leq\zeta(B)\cap C=0,$$ so $C$ is periodic. Thus, as a group, $C$ can be described as a direct product of infinitely many cyclic subgroups $\langle b_i\rangle$, $i\in I$, of order dividing $n$. 

Let $F$ be a finitely generated subbrace of $B$ such that $C+\zeta(B)+F=B$. Since $B$ is periodic, then it is locally finite by \cite{Tr23}, Lemma 3.1, so $F$ is finite. 

For every $b\in F$, $b_i^{b,+}=b_i+u_{b,i,+}$ for some $u_{b,i,+}\in\zeta(B)$. On the other hand, $\zeta(B)$ is finite, so there is an infinite subset $J_1$ of $I$ such that $$b_i^{b,+}=b_i+u_{+,b}$$ for all $i\in J_1$, and for a fixed $u_{+,b}\in\zeta(B)$. Repeating this argument for all $b\in F$, we may assume $$b_i^{b,+}=b_i+u_{+,b}$$ for some $u_{+,b}\in\zeta(B)$ and for all $i\in J_1$, $b\in F$. Similarly, there is an infinite subset~$J_2$ of $J_1$ such that $$b_i^{b,\cdot}=b_i+u_{\cdot,b}$$ for some $u_{\cdot,b}\in\zeta(B)$ and for all $i\in J_2$, $b\in F$. Finally, there is an infinite subset $J_3$ of $J_2$ such that $$b_i\ast b=b_j\ast b$$ for all $i,j\in J_3$ and $b\in F$.

Now, for each $i,j\in J_3$ with $i\neq j$, we have that $d_i=b_i-b_j=b_i\cdot b_j^{-1}$ is additively and multiplicatively centralised by $F$, and that $d_i\ast F=0$. Since $B=C+\zeta(B)+F$, it follows that $d_i\in\zeta(B)$ for all $i\in J_3$, a contradiction.
\end{proof}

\section{Central nilpotency for ideals}\label{sec:bnilp}

\noindent A celebrated result of Fitting states that a product of nilpotent normal subgroups of a group is nilpotent. Example~\ref{ex:nofitting} in Section~\ref{wex} shows that the product of two centrally nilpotent ideals, regarded as independent braces, is not centrally nilpotent in general. In this section, we aim to define a nilpotency concept for ideals that allows us to define a suitable Fitting ideal. It turns out that, for such a Fitting ideal, it is possible to generalise  remarkable results of group theory concerning with the Fitting subgroup. For example, we give a characterization of~Fit\-ting ideal in terms of chief factors (see~The\-o\-rem~\ref{teo:Fit_centralisers}), we prove that~Fit\-ting ideals are self-centralising in soluble braces (see~The\-o\-rem~\ref{prop:Fiting_autocent_soluble}), and we give an analogue of a well-known result of~Gaschütz stating that the~Fitting modulo the Frattini subgroup of a finite group is a product of all its minimal abelian normal subgroups (see~The\-o\-rem~\ref{gaschutz}); note that to prove the latter result to prove we need a suitable Frattini-like ideal for braces. In the final part of the section we discuss hypercentral (resp. locally nilpotent) concepts for ideals.

Let $B$ be a brace. We start by defining $B$-centrally nilpotent braces. Let~$I$ be an ideal of $B$. We can define the \emph{lower central series of $I$ with respect to~$B$}, or simply the \emph{lower $B$-central series of $I$}, as follows: take $\Gamma_1(I)^B = I$ and $\Gamma_{n+1}(I)^B = [\Gamma_n(I),I]^B$, for every $n \geq 1$. Therefore, 
\[ I=\Gamma_1(I)^B \geq \Gamma_2(I)^B \geq \ldots \geq \Gamma_n(I)^B \geq \ldots\]
is a descending chain of ideals of $B$ such that, for every $n\in \mathbb{N}$,
\[\Gamma_n(I)^B/\Gamma_{n+1}(I)^B \leq \zeta\big(I/\Gamma_{n+1}(I)^B\big).\] Similarly, we may define the {\it upper central series of $I$ with respect to $B$} (or simply the {\it upper $B$-central series of $I$}), as follows: take $\zeta_0(B)=0$ and let~$\zeta_{n+1}(I)^B$ satisfy $$\zeta_{n+1}(I)^B/\zeta_n(I)^B=\zeta\big(I/\zeta_n(I)^B\big)_{B/\zeta_n(I)^B}.$$ Then $$\zeta_0(I)^B\leq \zeta_1(I)^B\leq\ldots\leq\zeta_n(I)^B\leq\ldots$$ is an ascending chain of ideals of $B$.

\begin{definicio}\label{definitioncentralnilp}
An ideal $I$ of a brace $B$ is defined to be \emph{centrally nilpotent with respect of $B$}, or simply a \emph{$B$-centrally nilpotent ideal}, if there exists $n\in \mathbb{N}$ such that $\Gamma_{n+1}(I)^B = 0$, or, equivalently, $\zeta_n(I)^B=0$. For practical purposes, we often use the following equivalent definition: $I$ is $B$-centrally nilpotent if there exists a chain
\[ 0 = J_0 \leq J_1 \leq \ldots \leq J_n = I\]
of ideals of $I$ such that $J_i/J_{i-1}\leq\zeta(I/J_{i-1})$, for every $1\leq i \leq n$.

To simplify notation, if $J$ is an ideal of $B$ contained in $I$ and such that~$I/J$ is $B/I$-centrally nilpotent, we just say that $I/J$ is \emph{centrally nilpotent with respect of $B$}, or \emph{$B$-centrally nilpotent}. If $I/J$ is $B$-centrally nilpotent, then the smallest $n$ such that $\Gamma_{n+1}(I/J)^{B/J}=0$ is referred to as its {\it class}.
\end{definicio} 

Clearly, a brace $B$ is centrally nilpotent if and only if it is $B$-centrally nilpotent; in this trivial case, the {\it upper central series} (resp. {\it lower central series}) and the {\it upper $B$-central series} (resp. {\it lower $B$-central series}) coincide. Moreover, if $I$ is a $B$-centrally nilpotent ideal of a brace $B$, and $J$ is any ideal of $B$, then $(I+J)/J$ is $B$-centrally nilpotent (of class less than or equal to that of $I$), and $I\cap C$ is $C$-centrally nilpotent for any subbrace $C$ of $B$ (also in this case the class of $I\cap C$ is less than or equal to that of $I$). Our next result shows that an analog of Fitting theorem holds for $B$-centrally nilpotent ideals, but first, we need the following property of commutators of ideals in braces.

\begin{lema}\label{lema:propietats_commut}
Let $B$ be a brace and let $I,J,K$ be ideals of $B$. Then $$[I, JK]_B = [I, J+ K]_B = [I,J]_B + [I,K]_B  = [I,J]_B[I,K]_B.$$
\end{lema}
\begin{proof}
We prove the equality for the sum. Observe that only one inclusion is in doubt so, by Theorem~\ref{teo:commutador=ast}, it suffices to show that 
\[ [I,J+K]_+,\, I\ast (J+K),\, (J+K) \ast I\subseteq [I,J]_B + [I,K]_B.\]

Since $(I,+)$, $(J,+)$ and $(K,+)$ are normal subgroups of $(B,+)$, we have that $[I, J+K]_+ = [I, J]_+ + [I,K]_+$ is contained in $[I,J]_B + [I,K]_B$. Then, applying Eq.~\eqref{eq:identitat_astsum}, we have that for every $i\in I$, $j\in J$ and $k \in K$,
\[ i \ast (j+k) = i \ast j + j + i \ast k  - j \in [I,J]_B + [I,K]_B \]

Finally, note that $(J+K) \ast I = (JK) \ast I$, so applying Eq.~\eqref{eq:identitat_astprod} we see that \[ (jk) \ast i = j \ast (k\ast i) + k \ast i + j \ast i \in J \ast I + K \ast I + J \ast I \subseteq [I,J]_B + [I,K]_B\] for every $j\in J$, $k \in K$, $i \in I$.
\end{proof}

\medskip

We also need the following notation in the proof: if $I_1, \ldots, I_n$ are ideals of a brace $B$, we put $[I_1]^B = I_1$, and then, recursively, 
\[ [I_1, \ldots, I_k]^B := \big[[I_1, \ldots, I_{k-1}]^B, I_k\big]^B \quad \text{for every $2\leq k \leq n$.}\]

\begin{teorema}
\label{teo:Fitting-brace}
Let $I, J$ be $B$-centrally nilpotent ideals of a brace $B$. If $I$ and~$J$ have classes $n_0$ and $m_0$, respectively, then  $I+J$ is $B$-centrally nilpotent of class at most $n_0+m_0$.
\end{teorema}
\begin{proof}
Set $K = I+J$. First, we show by induction that for every $n\in \mathbb{N}$, $\Gamma_n(K)^B$ is the sum of all commutators of the form $[L_1, \ldots, L_n]_B$ with  either $L_i = I$ or $L_i = J$, for every $1\leq i \leq n$. The base case is clear. Assume the assertion is true for some $1\leq n\in \mathbb{N}$. Then,
\[ \Gamma_{n+1}(K)^B = [\Gamma_n(K),K]^B = [\Gamma_n(K),I]^B + [\Gamma_n(K), J]^B \]
by Lemma~\ref{lema:propietats_commut}. Using iteratively Lemma~\ref{lema:propietats_commut}, the assertion also holds for~$n+1$.

In particular, for $r = n_0 + m_0+1$, $\Gamma_{r}(K)^B$ is the sum of all commutators of the form $[L_1, \ldots, L_r]^B$, where either $I$ occurs $n_0+1$ times or $J$ occurs $m_0+1$ times. Thus, it follows that each $[L_1, \ldots, L_r]^B$ is contained in either $\Gamma_{n_0+1}(I)^B = 0$ or $\Gamma_{m_0+1}(J)^B = 0$. Hence, $\Gamma_{r}(K)^B = 0$ and so $K$ is~\hbox{$B$-central}\-ly nilpotent.
\end{proof}

\medskip

We will later generalise Theorem \ref{teo:Fitting-brace} to Theorem \ref{produhypercentral}.

\begin{definicio}
Let $B$ be a brace. The {\it Fitting ideal} $\Fit(B)$ of $B$ is the ideal generated by all $B$-centrally nilpotent ideals of $B$.
\end{definicio}

It follows from Theorem \ref{teo:Fitting-brace} that in a finite brace $B$, $\Fit(B)$ is $B$-centrally nilpotent. More in general, the same result shows that this is true for a broader class of braces. 

\begin{corolari}
\label{cor:noetherian_Fit}
Let $B$ be a noetherian brace. Then $\Fit(B)$ is a $B$-centrally nilpotent ideal.
\end{corolari}

\smallskip

Now, in order to obtain a characterisation of the Fitting ideal in terms of  chief factors, we need the following definition (recall Proposition~\ref{prop:centre-commut}). 
In \cite[Propostion 4.19]{BournFacchiniPompili23}, the \emph{centraliser of an ideal $I$} of a brace $B$, $\C_B(I)$, is defined as the largest ideal that {\it centralises} $I$, i.e. $[\C_B(I),I]^B = 0$.

Moreover, if $I/J$ is a chief factor of $B$, we define the \emph{centraliser in $B$ of $I/J$} as the ideal $\C_B(I/J)$ of $B$ satisfying $\C_{B/J}(I/J) = \C_B(I/J)/J$. Equivalently, $\C_B(I/J)$ is the largest ideal of $B$ such that $[\C_B(I/J), I]^B \leq J$.

\begin{lema}\label{idealcentralizzato}
Let $I$ be a $B$-centrally nilpotent ideal of a brace $B$. If $J$ is a minimal ideal of $B$, then $[J,I]^B=0$.
\end{lema}
\begin{proof}
Since $J$ is a minimal ideal of $B$, then either $[J,I]^B=0$ or $[J,I]^B=J$. However, in the latter case we contradict Definition \ref{definitioncentralnilp}.
\end{proof}

\begin{teorema}
\label{teo:Fit_centralisers}
Let $B$ be a brace with a finite chief series $\mathcal{S}$. Then $\Fit(B)$ is the intersection of the centralisers in $B$ of the factors of $\mathcal{S}$.
\end{teorema}
\begin{proof}
Let
\[0=I_0\leq I_1 \leq \ldots \leq I_n = B\]
be a finite chief series of $B$ and set $C := \bigcap \{ \C_B(I_k/I_{k-1}): 1 \leq k \leq n\}$. Then $C$ is an ideal of $B$ and $$0=C\cap I_0\leq C\cap I_1\leq\ldots\leq C\cap I_n=C$$ is a finite chain of ideals of $B$ such that $(I_i\cap C)/(I_{i-1}\cap C)\leq\zeta\big(C/I_{i-1}\cap C\big)$ for all $1\leq i\leq n$. Thus, $C$ is $B$-centrally nilpotent and hence $C\leq \Fit(B)$.

Conversely, $B$ is noetherian (see \cite{BallesterEstebanPerezC-JH}) and so $F:=\Fit(B)$ is $B$-centrally nilpotent by Corollary \ref{cor:noetherian_Fit}. If $I/J$ is any chief factor of $B$, then $I/J$ is centralised by $(F+J)/J$ by Lemma \ref{idealcentralizzato}.
\end{proof}

\begin{teorema}
\label{prop:Fiting_autocent_soluble}
Let $B$ be a brace and put $F = \Fit(B)$. Then, \hbox{$(\C_B(F)+F)/F$} does not contain any non-zero soluble ideal with respect of $B/F$. In particular, if $B$ is a soluble brace, then $\C_B(F) = \zeta(F)$.
\end{teorema}
\begin{proof}
Assume that $(\C_B(F) + F)/F$ contains a non-zero soluble ideal $I/F$  with respect of $B/F$. Then it also contains a non-zero ideal $J/F$ which is an abelian brace. Let $C=\C_B(F)$. Then $J\cap C\cap F\leq\zeta(J\cap C)$ and \hbox{$(J\cap C)/(J\cap C\cap F)$} is an abelian brace. Thus, $J\cap C$ is $B$-centrally nilpotent and so $J\cap C\leq F$. Finally, $$J=J\cap(C+F)=(J\cap C)+F=F,$$ a contradiction.

If $B$ is a soluble brace, then $B/F$ is soluble and therefore, $(\C_B(F) + F)/F$ must be zero. Hence, $\C_B(F) \leq F$ which yields $\C_B(F) = \zeta(F)$.
\end{proof}

It is well-known that the Fitting subgroup of a finite group modulo its~Frat\-tini subgroup is the product of all its abelian minimal normal subgroups. The following Frattini-like ideal leads to a brace-theoretic analogue of this result.

\begin{definicio}
Let $B$ be a finite brace. We define {\it Frattini ideal} of $F$ as \[ \operatorname{Frat}(B):= \bigcap\{\,L \,|\, \text{$L$ is a maximal left ideal of $B$}\,\} \cap \Fit(B).\]
\end{definicio}

\medskip

Clearly, the Frattini ideal of a finite brace is a left ideal, but the following result shows that it is actually an ideal (hence providing a justification for its name).

\begin{lema}\label{lemmax}
Let $B$ be a finite brace. If $L$ is any maximal left ideal of $B$, then $L\cap\operatorname{Fit}(B)$ is an ideal of $B$.
\end{lema}
\begin{proof}
Let $F=\operatorname{Fit}(B)$ and assume $F\not\leq L$, so in particular $B = FL$. 

Since $F\cap L \unlhd (L,+)$, it follows that $L \leq \N_{(B,+)}(F\cap L)$. Moreover, $F\cap L$ is properly contained in $\N_{(F,+)}(F\cap L)$, as $(F,+)$ is nilpotent. Therefore,  $L$ is properly contained in $\N_{(B,+)}(F\cap L)$.

Because $F\cap L$ is $\lambda$-invariant, it holds that $\N_{(B,+)}(F\cap L)$ is also $\lambda$-invariant. Thus, $\N_{(B,+)}(F\cap L) = B$ and so, $F\cap L$ is a strong left ideal. Then, by~Lem\-ma~\ref{annihilatorascendant} and Remark \ref{potenziare}, we can find a strong left ideal $T$ of $B$ contained in $F$ and such that $F\cap L$ is a proper ideal of $T$. Therefore, it follows that $F\cap L$ is  an ideal of $T+L=TL$, with $T+L$ being a left ideal of $B$. Hence, $TL = B$ by the maximality of $L$ and the result follows.
\end{proof}

\begin{corolari}
Let $B$ be a finite brace. Then $\operatorname{Frat}(B)$ is an ideal of $B$.
\end{corolari}

\begin{teorema}\label{gaschutz}
Let $B$ be a finite brace with $\Frat(B)=0$. Then $\Fit(B)$ is the product of all the abelian minimal ideals of $B$.
\end{teorema}

\begin{proof}
Let $F=\Fit(B)$. We claim that $\partial_B(F)=0$. Indeed, suppose that $L$ is a maximal left ideal of $B$ such that $\partial_B(F)$ is not included in $L$. Then, $B = \partial_B(F)L$, so
\[ F = F \cap \partial_B(F)L = (F \cap L)\partial_B(F).\] By Lemma \ref{lemmax}, $F\cap L$ is an ideal of $B$. Since $F/(F\cap L)$ is non-zero and~\hbox{$B$-cen}\-tral\-ly nilpotent, we have that $$\partial_B(F)(F\cap L)/(F\cap L)\leq \partial_{B/(F\cap L)}\big(F/(F\cap L)\big)<F/(F\cap L),$$ a contradiction. Thus, $F$ is an abelian brace.

Let $N$ be the product of all abelian minimal ideals of $B$. Then $N\leq F$. For the other inclusion, take $S$ a minimal subbrace subject to $B = SN$. Consider
\[ X = \bigcap \{L\,|\, \text{$L$ is a maximal left ideal of $S$}\},\]
a left ideal of $S$. If $S\cap N \not\leq X$, then there exists a maximal left ideal $L$ of $S$ such that $S \cap N$ is not included in $L$. Thus, $(S\cap N)L = S$ and then $B = SN = (S\cap N)LN = LN$, which contradicts the minimality of $S$. Therefore, $S \cap N \leq X$. 

Now, $S \cap N$ is an ideal of $B$, as $N$ is abelian and $B = SN$ (see~\cite[Lemma 27]{BallesterEstebanJimenezPerezC-solublebraces}). Assume that there exists a maximal left ideal $L$ of $B$ such that $S \cap N\not\leq L$. Thus, $B = (S\cap N)L$ and then, $S = S \cap (S\cap N)L = (S \cap L)(S\cap N)$. Take $L'$ a  left ideal of $S$ maximal subject to $S\cap L \leq L'$ and $S\cap N$ not included in $L'$. Then, $L'$ is indeed a maximal left ideal of $S$, because the existence of a left ideal $L''$ of $S$ such that $L'<L'' \leq S$ yields $S = (S\cap L)(S\cap N)\leq L''$. Therefore, $S\cap N \leq X\leq L'$, a contradiction. Thus, $S\cap N=0$.

Finally, $S\cap F$ is an ideal of $B$, as $F$ is abelian and $SF = B$ (see again~\cite[Lemma 27]{BallesterEstebanJimenezPerezC-solublebraces}), and consequently $S\cap F$ contains an abelian minimal ideal of~$B$, contradicting $S\cap N=0$.
\end{proof}

\medskip

In a centrally nilpotent brace, the Frattini ideal behaves pretty well. For example, it is possible to prove that it coincides with the set of non-generators.

\begin{definicio}
Let $B$ be a brace. We say that an element $b\in B$ is a {\it non-generator} of $B$ if for all $S\leq B$ such that  $B=\langle b,S\rangle$, we have $B=S$.
\end{definicio}

\begin{teorema}
Let $B$ be a centrally nilpotent finite brace. Then $\Frat(B)$ coincides with the set of all non-generators of $B$.
\end{teorema}
\begin{proof}
Since in a centrally nilpotent brace the maximal left ideals coincide with the maximal ideals and with the maximal subbraces, the usual group theoretical proof adapts to prove the following result.
\end{proof}

However, we note that there exist a brace $B$ of order $6$ in which $\Fit(B)=\Frat(B)$ is the only non-zero proper left ideal of $B$ and has order $3$. This shows that there is no possible analogue of two well-known group theoretical theorems concerning the Frattini subgroup of a group $G$: (1) $G/\Frat(G)$ nilpotent implies $G$ nilpotent; (2) if $p$ is a prime dividing $|G|$, then $p$ divides $|G/\Frat(G)|$ too.

\medskip

In the final part of this section we discuss further aspects of $B$-centrally nilpotence and hypercentral/locally nilpotent concepts for (sub)ideals.

\smallskip

The definition of upper $B$-central series (and lower $B$-central series) for an ideal $I$ of a brace $B$ can be extended by using transfinite numbers (just how we did in Section \ref{sec:nilp}), and this allows us to define $B$-hypercentral ideals of braces. However, for our purposes, the following equivalent definition is more convenient.

\begin{definicio}
Let $B$ be a brace. An ideal $I$ of $B$ is said to be {\it $B$-hy\-per\-cen\-tral} if there is an ascending chain of ideals of $B$ $$0=:I_0\leq I_1\leq\ldots I_\alpha\leq I_{\alpha+1}\leq\ldots I_\lambda=I$$ such that $I_{\alpha+1}/I_\alpha\leq\zeta(B/I_\alpha)$ for all ordinals $\alpha<\lambda$. The smallest $\lambda$ for which such a chain exists is the {\it length} of $I$. 
\end{definicio}

\medskip

Clearly, if $B$ is a brace, then $B$ is hypercentral if and only if $B$ is~\hbox{$B$-hyper}\-cen\-tral, and every $B$-centrally nilpotent ideal is $B$-hypercentral. The following result generalises Theorem \ref{teo:Fitting-brace}.

\begin{teorema}\label{produhypercentral}
Let $B$ be a brace.
\begin{itemize}
    \item[\textnormal{(1)}] If $C$ and $D$ are ideals of $B$ which are $B$-hypercentral of lengths $\alpha$ and~$\beta$, respectively, then $C+D$ is a~\hbox{$B$-hyper}\-central ideal of length at most $\beta\alpha+\operatorname{max}\{\alpha,\beta\}$.
    \item[\textnormal{(2)}] Suppose $C$ is subideal of defect $n$ and centrally nilpotent of class $c$, and~$D$ is a~\hbox{$B$-centrally} nil\-po\-tent ideal of class $d$. Then $C+D$ is centrally nilpotent of class at most $(c+n)d+c$.
    
    
    \item[\textnormal{(3)}] Suppose $C$ is ascendant of length $\mu$ and hypercentral of length $\gamma$, and $D$ is a $B$-hypercentral ideal of length $\delta$. Then $C+D$ is hypercentral of length at most $(\gamma+\mu)\delta+\gamma$.
\end{itemize}
\end{teorema}
\begin{proof}
(1)\quad Let $E=C+D$. Then $E$ is an ideal of $B$ and to show that $E$ is~\hbox{$B$-hyper}\-central, it suffices to prove that $\zeta(E)$ contains a non-zero ideal $I$ of $B$. To this aim we may certainly assume that $C$ and $D$ are non-zero.

Suppose first $C\cap D=0$. By hypothesis $\zeta(C)$ contains a non-zero ideal~$I$ of $B$. On the other hand, $\zeta(C)\leq\zeta(E)$ and so we are done.

Assume $C\cap D\neq0$. Then $\zeta(C)\cap D$ contains a non-zero ideal $I$ of~$B$, and $I\cap \zeta(D)$ contains a non-zero ideal $J$ of $B$. Thus $$J\leq \zeta(C)\cap\zeta(D)\leq\zeta(E)$$ and we are done. The bound on the hypercentral length can be easily deduced from the proof.

\medskip

\noindent(3)\quad Since $D$ is $B$-hypercentral of length $\delta$, there is an ascending chain of ideals of $B$ $$0=D_0<D_1<\ldots D_\alpha<D_{\alpha+1}<\ldots D_\delta=D$$ such that $D_{\beta+1}/D_\beta\leq\zeta(D/D_\beta)$ for all ordinals $\beta<\delta$.

Let $E=C+D$. Since $C$ is hypercentral of length $\gamma$, it follows that $C\cap D_1\leq\zeta_\gamma(E)$. Thus, we may factor out $C\cap D_1$ and assume $C\cap D_1=0$. Let  $F=\langle C,D_1\rangle=CD_1$. Now, since $C$ is ascendant of length $\mu$, there is an ascending chain $$C=C_{0}\trianglelefteq C_{1}\trianglelefteq\ldots C_{\alpha}\trianglelefteq C_{\alpha+1}\trianglelefteq\ldots C_{\mu}=F$$ connecting $C$ to $F$. It is easy to see that $$(C_{\beta+1}\cap D_{1})/(C_{\beta}\cap D_{1})\leq\zeta\big(E/(C_{\beta}\cap D_{1})\big)$$ for all $\beta<\mu$. Therefore $D_1\leq\zeta_\mu(E)$.

We factor $D_1$ out and we repeat the above argument. This shows that $D\leq\zeta_\rho(E)$, where $\rho=(\gamma+\mu)\delta$, so $E$ is hypercentral of length at most $\rho+\gamma$.

\medskip

\noindent(2)\quad The proof is essentially the same as that of (3), but easier.
\end{proof}

\medskip

\medskip

On some occasions, the product of all $B$-hypercentral ideals of a brace is still $B$-hypercentral. This is clearly the case for instance if $B$ satisfies the maximal condition on ideals, but it is also the case if $B$ satisfies the minimal condition on ideals, as the following result shows.

\begin{teorema}
Let $B$ be a brace admitting an ascending chief series $\mathcal{S}$. Then the maximal ideal centralising all factors of $\mathcal{S}$ is precisely the unique maximal $B$-hypercentral ideal of~$B$.
\end{teorema}
\begin{proof}
The proof runs along the same lines as that of Theorem \ref{teo:Fit_centralisers}.
\end{proof}

The following result shows that in the universe of locally centrally-nilpotent braces, the class of $B$-hypercentral braces is closed with respect to forming extensions by finitely generated hypercentral braces.

\begin{teorema}
Let $N$ be an ideal of the locally centrally-nilpotent brace~$B$. If~$N$ is $B$-hypercentral and~$B/N$ is finitely generated, then $B$ is~hypercentral.
\end{teorema}
\begin{proof}
Let $S$ be a finite subset of $B$ that generates $B$ modulo $N$, and let $Z=\overline{\zeta}(B)$ be the hypercentre of $B$. Assume by contradiction $Z\neq B$. Now,~$B/N$ is centrally nilpotent, so $N\not\leq Z$ and hence the ideal $K:=Z\cap N$ of $G$ is strictly contained in $N$. Since $N/K$ is $B$-hypercentral, there is a non-zero ideal $A/K$ of $B/K$ such that $$A/K\leq\zeta(N/K).$$ Let $a\in A\setminus K$ and $U=\langle a,S,K\rangle$; in particular, $U/K$ is centrally nilpotent. Since \hbox{$(A\cap U)/K$} is a non-zero ideal of $U/K$, we have that $$V/K:=\zeta(U/K)\cap\big((A\cap U)/K\big)\neq0.$$ Now, the fact that $V\leq A$ implies that $[V,N]_+$, $[V,N]_\cdot$ and $V\ast N$ are all contained in~$K$. Similarly, the fact that $V/K\leq\zeta(U/K)$ shows that $[V,T]_+$, $[V,T]_\cdot$ and $V\ast T$ are all contained in $K$, where $T=\langle S\rangle$. Since $B=N+T=NT$, we easily see that $[V,B]_+$ and $[V,B]_\cdot$ are contained in $K$. Moreover, if $u\in N$, $v\in T$, and $a\in V$, then 
$$a\ast (u+v)=a\ast u+u+a\ast v-u\in K.$$ This shows that $V\ast B\leq K$ and proves that $V/K\leq\zeta(B/K)$. Since $K\leq Z$, it follows that $V\leq Z$, so $V\leq Z\cap N$, a contradiction.
\end{proof}

\medskip

Also $B$-central nilpotency (resp. hypercentrality) can be locally detected, and our next result is in fact a generalisation of Theorem \ref{countablenilp}.

\begin{teorema}
Let $B$ be a brace and let $I\trianglelefteq B$. Then:
\begin{itemize}
\item[\textnormal{(1)}] $I$ is $B$-centrally nilpotent of class at most $c$ if and only if $I\cap F$ is $F$-centrally nilpotent of class at most $c$ for every finitely generated subbrace~$F$ of $B$.
\item[\textnormal{(2)}] $I$ is $B$-centrally nilpotent if and only if $I\cap C$ is $C$-centrally nilpotent for every countable subbrace $C$ of $B$.
\item[\textnormal{(3)}] $I$ is $B$-hypercentral if and only if $I\cap C$ is $C$-hypercentral for every countable subbrace $C$ of $B$.
\end{itemize}
\end{teorema}
\begin{proof}
We only deal with the proof of (1), since (2) and (3) then follow in a similar fashion using ideas from Theorem \ref{countablenilp}.

For each $u,v\in B$, we write $u\circ v$ to denote one (but we do not know which one) of the following operations $[u,v]_\cdot,[u,v]_+,u\ast v$. Then (1) is a direct consequence of the fact that $\zeta_c(I)^B$ can be easily characterised as the set of all elements $b\in I$ such that $$\langle (\ldots((b\circ b_1)\circ\ldots)\circ b_{c-i})\rangle^B\leq\zeta_i(I)$$ for all $i=0,1,\ldots,c-1$ and for all $b_1,\ldots,b_{c-i}\in B$. 
\end{proof}

\bigskip

To provide a good definition of ‘‘locally $B$-nilpotent ideal’’ is not an obvious task. We now concern ourselves with a couple of possible definitions, mostly sketching proofs and results. The first idea that comes in mind is that of using central chains of ideals, just as we did for $B$-hypercentral and $B$-centrally nilpotent ideals. In fact, it follows from Theorem \ref{locallnilp} (and Zorn's lemma) that in every chief series of a locally centrally-nilpotent brace $B$, two consecutive ideals $K\leq H$ satisfy $H/K\leq\zeta(B/K)$.

\begin{definicio}
{\rm Let $B$ be a brace. An ideal $I$ of $B$ is said to be {\it $\zeta_B$-nilpotent} if every quotient $I/J$ of $I$ by an ideal $J$ of $B$ admits a maximal chain $\mathcal{S}$ of ideals of $B$ in which $H/K\leq\zeta(I/K)$ for all consecutive terms $K/J\leq H/J$ of the chain $\mathcal{S}$.}
\end{definicio}

\medskip

Using ideas from the proof of Theorem \ref{produhypercentral}, it is not difficult to see that the product of arbitrarily many $\zeta_B$-nilpotent ideals is $\zeta_B$-nilpotent. Thus, any brace $B$ has a unique maximal $\zeta_B$-nilpotent ideal, we call it the {\it $\zeta_B$-radical} of~$B$: it turns out that the largest ideal of $B$ centralising all quotients of a chief series of $B$ is precisely the $\zeta_B$-radical of $B$. The following result shows that an analogue of Theorem \ref{prop:Fiting_autocent_soluble} is possible for the $\zeta_B$-radical.

\begin{teorema}
Let $B$ be a brace admitting an ascending chain of  ideals $$0=B_0\leq B_1\leq\ldots\ldots B_\alpha\leq B_{\alpha+1}\leq\ldots B_\lambda=B$$ such that $B_{\beta+1}/B_\beta$ is a $\zeta_B$-nilpotent ideal of $B/B_\beta$ for all $\beta<\lambda$. Then $C_B(H)\leq H$, where $H$ is the $\zeta_B$-radical of $B$. 
\end{teorema}
\begin{proof}
Suppose $C=C_B(H)\not\leq H$. Then $(C+H)/H$ contains a non-zero $\zeta_B$-nilpotent ideal $I/H$ of $B/H$. Since $I\cap C\cap H\leq\zeta(I\cap C)$, it follows that $I\cap C$ is a $\zeta_B$-nilpotent ideal of $B$. Thus, $I\cap C\leq H$ and $$I=I\cap (C+H)=(I+C)\cap H=H,$$ a contradiction.
\end{proof}

However, there is one reason for which this is not a really convincing good definition of ‘‘locally $B$-nilpotent ideal’’: a $\zeta_B$-ideal could not be locally centrally-nilpotent (there are examples even among groups) --- although  if~$B$ is locally finite, then a $\zeta_B$-ideal is locally centrally-nilpotent.

A more fruitful approach could deal with finitely generated subbraces and the way the ideal embeds into them. There are several ways in which this case be achieved, but the most reasonable one seems to be the following.

\begin{definicio}
{\rm Let $B$ be a brace. An ideal $I$ of $B$ is {\it locally $B$-nilpotent} if the following property holds:
\begin{itemize}
    \item For every finitely generated subbrace $F$ of $B$, the finitely generated subbraces of $I\cap F$ are contained in $F$-centrally nilpotent ideals of~$F$.
\end{itemize}}
\end{definicio}

Trivially, every locally $B$-nilpotent ideal is locally centrally-nilpotent, so this solves the previous issue for $\zeta_B$-nilpotency.

\begin{teorema}\label{hirsch}
Let $B$ be a brace. The sum of arbitrarily many locally~\hbox{$B$-nil}\-potent ideals of $B$ is locally $B$-nilpotent.
\end{teorema}
\begin{proof}
It is clearly enough to prove the statement for two locally $B$-nilpotent ideals~$I$ and $J$. Let $F$ be a finitely generated subbrace of $B$. Choose a finitely generated subbrace $E$ of $F\cap (I+J)$. In order to prove that $E$ is contained in an $F$-centrally nilpotent ideal of $F$, we may assume $E=E_1\cup E_2$, where $E_1\subseteq I$ and $E_2\subseteq J$, by suitably replacing $F$. Now, $E_1$ and $E_2$ are respectively contained in $F$-centrally nilpotent ideals $I_1$ and $I_2$ of $F$. Since $I_1+I_2$ is \hbox{$F$-cen}\-tral\-ly nilpotent, we are done.
\end{proof}

\medskip

By Theorem \ref{hirsch}, every brace admits a unique maximal locally $B$-nilpo\-tent ideal, we call it the {\it Hirsch-Plotkin radical} of $B$, and we denote it by~$\operatorname{HP}(B)$. Using ideas from the proof of Theorem \ref{locallnilp} (1), we see that every locally $B$-nilpotent ideal is actually $\zeta_B$-nilpotent. Finally, we note that for a locally finite brace $B$, the concepts of locally $B$-nilpotent ideal and $\zeta_B$-ideal coincide.

\section{Worked examples}
\label{wex}

\noindent In this section we describe the main examples of the paper. These examples are all constructed in a similar fashion, which we now explain, and all the computations can be done with the computer algebra system \textsf{GAP} \cite{GAP4-12-2} and the functions of its package \textsf{YangBaxter} \cite{VendraminKonovalov22-YangBaxter-0.10.2}.

Braces can be defined from bijective $1$-cocycles associated with actions of groups. Let $(C,\cdot)$ and $(B,+)$ be groups such that $C$ acts on $B$ by means of a group homomorphsim $\lambda \colon C \rightarrow \Aut(B,+)$, $c\mapsto \lambda_c$. A bijective map $\delta \colon C \rightarrow A$ is said to be a \emph{bijective $1$-cocycle} associated with $\lambda$, if $\delta(cd) = \delta(c) + \lambda_c(\delta(d))$, for every $c,d\in C$. Following \cite{BallesterEsteban22}, bijective $1$-cocycles can be constructed by means of trifactorised groups. In the previous situation, take $G = [B]C$ the semidirect product associated with $\lambda$, written in multiplicative notation for the sake of uniformity. If $D$ is a subgroup of $G$ such that $G = BC = BD = DC$ and $B \cap D = D\cap C = 1$, then $G$ is said to be a \emph{trifactorised group}, and there exists a bijective $1$-cocycle $\delta \colon C \rightarrow (B,+)$, given by $D = \{\delta(c)c: c\in C\}$. At this point, observe that $\delta(c)$ must be translated in additive notation. Then, $(B,+)$ admits a brace structure by means of $ab := \delta(\delta^{-1}(a)\delta^{-1}(b))$, for every $a,b\in B$ (see \cite[Proposition 1.11]{GuarnieriVendramin17}), for example).

\smallskip

We are now ready to delve into our examples. The first of them shows that the idealiser of a subbrace (as introduced in \cite{JespersVanAntwerpenVendramin22}) does not exist in general (even for braces of abelian type).

\begin{aex}
\label{ex:noidealiser}
Let $B =\langle a_1\rangle \times \langle a_2\rangle \times \langle a_3\rangle\cong C_4\times C_4\times C_2$, whose operation will be written additively, and let\[
\begin{array}{c}
  C=\langle m_1,m_2,m_3,m_4,m_5\mid m_1^2=m_4, m_2^2=1, m_3^2=m_4^2=m_5, m_5^2=1,\\[0.1cm]
                m_1m_2{m_1^{-1}}=m_3m_2,
                           m_1m_3{m_1^{-1}}=m_5m_3, m_2m_3m_2^{-1}=m_5m_3,\\[0.1cm]
                    m_1m_4{m_1^{-1}}=m_2m_4{m_2^{-1}}=m_3m_4{m_3^{-1}}=m_4,\\[0.1cm]
  m_1m_5{m_1^{-1}}=m_2m_5{m_2^{-1}}=m_3m_5{m_3^{-1}}=m_4m_5{m_4^{-1}}=m_5\rangle.
\end{array}\]
We note that $B$ and $C$ are groups of order~$32$. Since $m_3=m_1m_2m_1^{-1}m_2^{-1}$, $m_4=m_1^2$, $m_5=m_1m_3m_1^{-1}m_3=m_3^2$, we have that $\langle m_3, m_4, m_5\rangle$ is contained in $\Frat(C)$ (in fact, they coincide) and so $C=\langle m_1, m_2\rangle$. We have that $C$ acts on $B$ by means of an action $\lambda$ defined by
\begin{align*}
  \lambda_{m_1}(a_1)&=3a_1+a_3,&\lambda_{m_2}(a_1)&=3a_1,\\
  \lambda_{m_1}(a_2)&=a_1+a_2+a_3,&\lambda_{m_2}(a_2)&=a_1+a_2+a_3,\\
  \lambda_{m_1}(a_3)&=a_3,&\lambda_{m_2}(a_3)&=a_3.
\end{align*}
We note that
\begin{align*}
  \lambda_{m_3}(a_1)&=\lambda_{m_1m_2m_1^{-1}m_2^{-1}}(a_1)=a_1,&\lambda_{m_4}(a_1)&=\lambda_{m_1^2}(a_1)=a_1\\
  \lambda_{m_3}(a_2)&=\lambda_{m_1m_2m_1^{-1}m_2^{-1}}(a_2)=a_2+a_3,&\lambda_{m_4}(a_2)&=\lambda_{m_1^2}(a_2)=a_2+a_3,\\
  \lambda_{m_3}(a_3)&=\lambda_{m_1m_2m_1^{-1}m_2^{-1}}(a_3)=a_3,&\lambda_{m_4}(a_3)&=\lambda_{m_1^2}(a_3)=a_3,
\end{align*}
and $\lambda_{m_5}$ is the identity map on $B$.

We can consider the semidirect product $G=[B]C$ with respect to this action. Then $G$ turns out to be a trifactorised group as it possesses a subgroup $D=\langle a_1a_2^3m_1, a_1m_2\rangle$ such that $D\cap C=D\cap B=1$, $DC=BD=G$. Thus, there exists a bijective $1$-cocycle $\delta\colon C\longrightarrow B$ with respect to $\lambda$ given by Table~\ref{tb:noidealiser}). This yields a product $\cdot$ in $B$ and we get a brace of abelian type $(B,+,\cdot)$ of order~$32$. This brace corresponds to \hbox{\texttt{SmallBrace(32, 14649)}} in the \textsf{Yang--Baxter} library for~\textsf{GAP}.
\begin{table}
  \[
    \begin{array}{llll}
      c&\delta(c)&c&\delta(c)\\\hline
    1&0                      &m_1&a_1+3a_2               \\
    m_5&2a_1                 &m_1m_5&3a_1+3a_2         \\
    m_4&3a_1+2a_2+a_3         &m_1m_4&a_2                \\
    m_4m_5&a_1+2a_2+a_3       &m_1m_4m_5&2a_1+a_2       \\
    m_3&3a_1+a_3              &m_1m_3&2a_1+3a_2          \\
    m_3m_5&a_1+a_3            &m_1m_3m_5&3a_2           \\
    m_3m_4&2a_1+2a_2          &m_1m_3m_4&a_1+a_2         \\
    m_3m_4m_5&2a_2            &m_1m_3m_4m_5&3a_1+a_2    \\
    m_2&a_1                   &m_1m_2&3a_2+a_3           \\
    m_2m_5&3a_1              &m_1m_2m_5&2a_1+3a_2+a_3  \\
    m_2m_4&2a_2+a_3           &m_1m_2m_4&3a_1+a_2+a_3    \\
    m_2m_4m_5&2a_1+2a_2+a_3  &m_1m_2m_4m_5&a_1+a_2+a_3  \\
    m_2m_3&2a_1+a_3           &m_1m_2m_3&3a_1+3a_2+a_3   \\
    m_2m_3m_5&a_3             &m_1m_2m_3m_5&a_1+3a_2+a_3 \\
    m_2m_3m_4&a_1+2a_2        &m_1m_2m_3m_4&2a_1+a_2+a_3 \\
    m_2m_3m_4m_5&3a_1+2a_2   &m_1m_2m_3m_4m_5&a_2+a_3   \\\hline
  \end{array}
  \]
  \caption{Associated bijective $1$-cocycle}
  \label{tb:noidealiser}
\end{table}

We have that $\langle 2a_1+2a_2\rangle_+\le (B,+)$, corresponding to $\langle m_3m_4\rangle\le C$ (through $\delta$), defines a subbrace $S$ of $B$ of order~$2$. Also, $\langle 2a_1, 2a_2, a_1+a_2+a_3\rangle_+\le (B,+)$, corresponding to $\langle m_5,m_3m_4m_5,  m_1m_2m_4m_5\rangle\le C$, defines a subbrace $T$ of~$B$ of order~$8$. Furthermore $\langle 2a_1,2a_2, a_1+a_3\rangle_+\le (B,+)$, corresponding to $\langle m_5, m_3m_4m_5, m_3m_5\rangle\le C$, defines another subbrace $U$ of~$B$ of order~$8$.

We note that $S$ is not a left ideal of $B$, because $\lambda_{m_1}(2a_1+2a_2)=2(3a_1+a_3)+2(a_1+a_2+a_3)=2a_2\notin S$. On the other hand, $S$ is a left ideal of $T$, since $\lambda_{m_5}(2a_1+2a_2)=\lambda_{m_3m_4m_5}(2a_1+2a_2)=\lambda_{m_1m_2m_4m_5}(2a_1+2a_2)=2a_1+2a_2$. Furthermore, $\langle m_3m_4\rangle$ is a normal subgroup of $\langle m_5, m_3m_4m_5, m_1m_2m_4m_5\rangle$. Therefore, $S$ is an ideal of $T$. We also have that $S$ is a left ideal of $U$, since $\lambda_{m_5}(2a_1+2a_2)=\lambda_{m_3m_4m_5}(2a_1+2a_2)=\lambda_{m_3m_5}(2a_1+2a_2)=2a_1+2a_2$. Moreover, $\langle m_3m_4\rangle$ is a normal subgroup of $\langle m_5, m_3m_4m_5,m_3m_5\rangle$. Therefore,~$S$ is an ideal of $U$.

We prove now that the subbrace $D=\langle T,U\rangle$ of $B$ generated by $T$ and $U$ is $B$.  Let $H$ be the additive group of $D$. Then $H\geq\langle 2a_1, a_2, a_1+a_3\rangle_+$. Thus, if $R=\delta^{-1}(H)$ is the corresponding multiplicative group, then $\delta^{-1}(2a_1)=m_5\in R$, $\delta^{-1}(a_2)=m_1m_4\in R$, $\delta^{-1}(2a_2)=m_3m_4m_5\in R$, $\delta^{-1}(a_1+a_3)=m_3m_5\in R$, $\delta^{-1}(a_1+a_2+a_3)=m_1m_2m_4m_5\in R$, which implies that $C=\langle m_1,m_2,m_3,m_4,m_5\rangle=R$. Thus, $H=(B,+)$ and $\langle T,U\rangle=B$.

Finally, suppose that $S$ possesses an idealiser in $B$. Since it must contain every subbrace of $B$ in which $S$ is an ideal, it must contain $T$ and $U$. It follows that the idealiser of $S$ in $B$ must be $B$, but $S$ is not an ideal of $B$.
\end{aex} 

\medskip

Our second example shows that there is no analogue of Fitting's theorem for central nilpotency, even for braces of abelian type

\begin{aex}
\label{ex:nofitting}
Let $B =\langle a\rangle\times \langle c\rangle \times \langle d\rangle \times \langle e\rangle\cong C_4\times C_2\times C_2\times C_2$, written additively. Let us consider the automorphisms $g_1$, $g_2$, $g_3$ of $B$ given by
\begin{align*}
g_1\colon a&\longmapsto 3a+d&g_2\colon a&\longmapsto a&g_3\colon a&\longmapsto a\\
c&\longmapsto c&c&\longmapsto c&c&\longmapsto 2a+c\\
d&\longmapsto d&d&\longmapsto 2a+d&d&\longmapsto d\\
e&\longmapsto c+e&e&\longmapsto 2a+e&e&\longmapsto 2a+e
\end{align*}
If $g_4=g_1g_2g_1^{-1}g_2^{-1}$ and $g_5=g_1g_3g_1^{-1}g_3^{-1}$, then their action on $(B,+)$ is
\begin{align*}
  g_4\colon a&\longmapsto 3a&g_5\colon a&\longmapsto a\\
  c&\longmapsto c&c&\longmapsto c\\
  d&\longmapsto d&d&\longmapsto d\\
  e&\longmapsto e&e&\longmapsto 2a+e
\end{align*}
and we have that $C=\langle g_1,g_2,g_3\rangle=\langle g_1, g_2, g_3, g_4, g_5\rangle$ satisfies the following relations:
\begin{align*}
  g_1^2&=1,\\
  g_2^2&=1,&g_1g_2g_1^{-1}&=g_4g_2,\\
  g_3^2&=1,&g_1g_3g_1^{-1}&=g_5g_3,&g_2g_3g_2^{-1}&=g_3,\\
  g_4^2&=1,&g_1g_4g_1^{-1}&=g_4,&g_2g_4g_2^{-1}&=g_4,&g_3g_4g_3^{-1}&=g_4,\\
  g_5^2&=1,&g_1g_5g_1^{-1}&=g_5,&g_2g_5g_2^{-1}&=g_5,&g_3g_5g_3^{-1}&=g_5,&g_4g_5g_4^{-1}&=g_5.
\end{align*}
It follows that $C$ is a group of order~$32$. We can consider the semidirect product $G = [B]C$ with respect to this action $\lambda \colon C\longrightarrow\operatorname{Aut}(B)$. Let
\[D=\langle a^2cg_1, cdeg_2, a^3cg_3\rangle=\langle a^2cg_1,cdeg_2,a^3cg_3,cg_4,dg_5\rangle; \]
then $D$ is a group of order~$32$, satisfying the same relations as~$C$, and, since $\langle a^2c,cde,a^3c,c,d\rangle=B$, we obtain that $G=DC=BD$ and $D\cap C=D\cap B=1$. This leads to the bijective $1$-cocycle $\delta\colon C \rightarrow B$ given in  Table~\ref{tab-exempleFitting}, and we can define a structure of a brace of abelian type on $(B,+,\cdot)$ of order~$32$. This brace corresponds to \texttt{SmallBrace(32, 23060)} in the \textsf{Yang--Baxter} library for \textsf{GAP}.
\begin{table}
\[
  \begin{array}{rlrl}
    x&\delta(x)    &        x&\delta(x)  \\\hline
    1&0            &g_5g_3g_2&a+e     \\
    g_1&2a+c       &g_3g_2g_1&3a+e    \\
    g_2&c+d+e      &g_4g_2g_1&2a+d+e  \\
    g_3&3a+c       &g_5g_2g_1&c+e     \\
    g_4&c          &g_4g_3g_1&a+c+d   \\
    g_5&d          &g_5g_3g_1&3a      \\
    g_2g_1&2a+c+d+e&g_5g_4g_1&2a+d    \\
    g_3g_1&3a+d    &g_4g_3g_2&a+c+d+e \\
    g_4g_1&2a   & g_5g_4g_2&2a+e        \\
    g_5g_1&2a+c+d& g_5g_4g_3&a+d           \\
    g_3g_2&3a+d+e& g_4g_3g_2g_1&a+c+e      \\
    g_4g_2&d+e   & g_5g_3g_2g_1&a+d+e      \\
    g_5g_2&2a+c+d& g_5g_4g_2g_1&e          \\
    g_4g_3&a     & g_5g_4g_3g_1&a+c        \\
    g_5g_3&3a+c+d& g_5g_4g_3g_2&3a+c+e     \\
    g_5g_4&c+d   & g_5g_4g_3g_2g_1&3a+c+d+e \\\hline
  \end{array}\]
  \caption{Associated bijective $1$-cocycle}
  \label{tab-exempleFitting}
\end{table}

Since $C$ has order~$32$, we have that $\operatorname{Ker}\lambda=1$. In particular, we have that $\zeta(B)=0$, and $B$ is not centrally nilpotent.

Now, let us compute the ideals of~$B$. Suppose that $I$ is a non-zero ideal of $B$ with additive group $L$ and multiplicative group $E$. Since $E$ is a normal subgroup of $C$, it must contain a minimal normal subgroup of $E$. All minimal normal subgroups of $C$ are contained in $\operatorname{Z}(C)=\langle g_4, g_5\rangle$. Hence $E$ must contain $\langle g_4\rangle$, $\langle g_5\rangle$ or $\langle g_4g_5\rangle$. In the first case, $L$ must contain $\delta(g_4)=c$. Since $L$ must be invariant under the action of $C$, it should contain $g_3(c)=2a+c$. Consequently, $\langle 2a,c\rangle_+\le L$. In particular, $\delta^{-1}(2a)=g_4g_1\in E$ and~\hbox{$\langle g_1, g_4\rangle\le E$.} Since $E\trianglelefteq C$, we have that $g_3g_1g_3^{-1}=g_5g_1\in E$, so $\langle g_1, g_4, g_5\rangle\le E$. 

Similarly, if $g_5\in E$, then $\delta(g_5)=d\in L$. Thus, $g_2(d)=2a+d\in L$ and hence $\langle 2a,d\rangle\leq L$. Now, $g_1g_4\in E$, so $g_3^{-1}g_1g_4g_3=g_1g_4g_5\in E$ and also $g_5\in E$. Therefore $\langle g_1, g_4, g_5\rangle\le E$.

Finally, if $g_4g_5\in E$, then $\delta(g_4g_5)=c+d\in L$, so $g_2(c+d)=2a+c+d\in L$ and hence $\langle 2a,c+d\rangle_+\leq L$. Again, $g_1g_4\in E$, so $g_3^{-1}g_1g_4g_3=g_1g_4g_5\in E$ and also $g_5\in E$. Thus, $\langle g_1, g_4, g_5\rangle\le E$.

In all cases, we found out that $\langle g_1, g_4, g_5\rangle\le E$. Since $\delta(\langle g_1,g_4,g_5\rangle)=\langle 2a, c,d\rangle_+\leq (B,+)$ is a $\delta$-invariant subgroup and $\langle g_1,g_4,g_5\rangle\trianglelefteq C$, we have that $J=\langle 2a,c,d\rangle_+$ is the unique ideal of $B$ of order~$8$. We observe that~$B/J$ is abelian. Therefore, the only three ideals of order $16$ of $B$ are $I_1=\langle 2a, c, d, e\rangle_+$,  $I_2=\langle a+e, 2a, c, d\rangle_+$, $I_3=\langle a, c,d\rangle_+$.

It can be easily checked that
\[0\leq \langle c\rangle\le \langle 2a, c\rangle_+\le I_1,\quad 0\le \langle c+d\rangle_+ \le \langle 2a, c+d\rangle_+ \le I_2\]  and
\[0\le \langle d\rangle_+ \le \langle 2a, d\rangle_+\le I_3\] are $c$-series of $I_1$, $I_2$ and $I_3$, respectively. In particular, $I_1$, $I_2$ and $I_3$ are centrally nilpotent braces. However, $B=I_1+I_2=I_1+I_3=I_2+I_3$, but, as we have mentioned, $B$ is not centrally nilpotent.
\end{aex}

\medskip

Our third example shows that another classical property of nilpotency of finite groups fails for central nilpotency in finite braces: there may be abelian subideals that are not contained in any centrally nilpotent ideals.

\begin{aex}
\label{ex:triv_nosubnilid}
Let $$(B,+) =\langle a \rangle \times \langle b \rangle \cong \operatorname{C}_2\times \operatorname{C}_{12}\quad\textnormal{ and }\quad(C,\cdot) = [\langle \sigma\rangle]\langle\tau\rangle \cong \operatorname{Dih}_{24}.$$ We have that $C$ acts on $B$ by means of the action $\lambda$ defined by
\begin{align*}
  \lambda_{\sigma}(a)&= a +6b,&\lambda_{\tau}(a)&=a,\\
  \lambda_{\sigma}(b)&= a+b,&\lambda_{\tau}(b)&= a -b,
\end{align*}
We can consider the semidirect product $G$ of $B$ and $C$ with respect to this action. Then $G$ turns out to be a trifactorised group as it possesses a subgroup $D=\langle ab^7\sigma, b^6\tau\rangle$ such that $D\cap C=D\cap B=\{1\}$, $DC=BD=G$. Thus, there is a bijective $1$-cocycle $\delta\colon C\longrightarrow B$ with respect to $\lambda$ given by  Table~\ref{tb:triv_nosubnilid}. This yields a product in $B$ and we get a brace of abelian type $(B,+,\cdot)$ of order $24$. This brace corresponds to \texttt{SmallBrace(24, 57)} in the \textsf{Yang--Baxter} library for \textsf{GAP}.
\begin{table}
  \[
    \begin{array}{llllllll}
      c&\delta(c)&c&\delta(c)& c& \delta(c)& c& \delta(c)\\\hline
    1&0                 & \sigma^6 & a        & \tau         & 6b   & \sigma^6\tau    & a+6b \\
    \sigma & a+7b       & \sigma^7 & b        & \sigma\tau   & a+b  & \sigma^7\tau    & 7b\\
    \sigma^2 & a + 8b   & \sigma^8 & 8b       & \sigma^2\tau & a+2b & \sigma^8\tau    & 2b \\
    \sigma^3 & 9b       & \sigma^9 & a+3b     & \sigma^3\tau & 3b   & \sigma^9\tau    & a+9b\\
    \sigma^4 & 4b       & \sigma^{10} & a+4b  & \sigma^4\tau & 10b  & \sigma^{10}\tau & a+10b  \\
    \sigma^5 & 11b      & \sigma^{11} & 5b    & \sigma^5\tau & a+5b & \sigma^{11}\tau & 11b\\\hline
  \end{array}
  \]
  \caption{Associated bijective $1$-cocycle}
  \label{tb:triv_nosubnilid}
\end{table}

Let $I$ be any ideal of $B$ of order $12$ put $E=\delta^{-1}(I,+)$. Since $(I,+)$ is a maximal subgroup of $(B,+)$, it must contain its Frattini subgroup, which is~$\langle 6b\rangle$. As $\delta^{-1}(6b) = \tau$ and $E\trianglelefteq C$, it follows that $\sigma\tau\sigma^{-1} = \sigma^2\tau \in E$. Therefore, $\delta(\sigma^2\tau) = a+2b \in I$ and then $(I,+) = \langle a, 2b\rangle_+$. Since $I$ is~\hbox{$\lambda$-in}\-va\-riant, we get that $I = \langle a,2b\rangle_+$ is the only ideal of order $12$ of~$B$. 

Observe that $I$ is not abelian as $\Soc(I) = \langle a+4b\rangle_+$. Therefore, $(I,\cdot)$ is isomorphic to $\operatorname{Dih}_{12}$ and so $I$ can not be centrally nilpotent. Hence, $\Soc(I)$ is an abelian  subideal of $B$ of order $6$ such that it is not contained in any centrally nilpotent ideal of~$B$. 
\end{aex} 

Our last example shows that another sufficient condition for finitely generated nilpotent groups also falls in braces: there may be non-centrally nilpotent braces such that all subbraces are subideals.

\begin{aex}\label{ex:D}
Let $(B,+,\cdot)$ be the brace of abelian type of order $32$ studied in~\cite[Example 37]{BallesterEstebanJimenezPerezC-solublebraces}. It corresponds to \texttt{SmallBrace(32,24003)} in the \textsf{Yang--Baxter} library~for \textsf{GAP}, so that
\[\begin{array}{c}
(B,+) = \langle a \rangle \times \langle b\rangle \times \langle c \rangle \times \langle d \rangle \cong C_4\times C_2\times C_2 \times C_2,\quad\textnormal{and}\\[0.1cm]
(B,\cdot)  \cong \langle e, f, h \rangle \cong [C_2\times \operatorname{Q}_8]C_2
\end{array}\]
with bijective $1$-cocycle given by Table~\ref{tab:subid-nonilp}

\begin{table}[h]
\small
  \[
  \begin{array}{llllllll}
    x&\delta(x)&x&\delta(x)&x&\delta(x)&x&\delta(x)\\\hline
    1&0&h&c&f&a& fh&a+c\\
    e&3a+b&
    eh&3a&
    ef&b+c+d&
    efh&c+d\\
    e^2&b+c&
    e^2h&2a+b+c+d&
    e^2f&3a+d&
    e^2fh&a+c+d\\
    e^3&3a+b+d&
    e^3h&a+d&
    e^3f&b&
    e^3fh&2a\\
    e^4&2a+b+c&
    e^4h&2a+b&
    e^4f&a+b+c&
    e^4fh&a+b\\
    e^5&3a+c&
    e^5h&3a+b+c&
    e^5f&2a+d&
    e^5fh&2a+b+d\\
    e^6&2a+c+d&
    e^6h&d&
    e^6h&3a+b+c+d&
    e^6fh&a+b+d\\
    e^7&3a+c+d&
    e^7h&a+b+c+d&
    e^7f&2a+c&
    e^7fh&b+d\\\hline
  \end{array}\]
  \caption{Associated bijective $1$-cocycle}\label{tab:subid-nonilp}
\end{table}
and associated action given by
\begin{align*}
  e\colon a&\longmapsto a+c+d&f\colon a&\longmapsto a+b+c&h\colon a&\longmapsto a\\
  b&\longmapsto 2a+c&b&\longmapsto 2a+b& b&\longmapsto b\\
  c&\longmapsto b&c&\longmapsto c&c&\longmapsto c\\
  d&\longmapsto 2a+c+d&d&\longmapsto c+d&d&\longmapsto 2a+d.
\end{align*}

We start by providing all subbraces of order $2$. These are generated by those elements $x\in (B,+)$ of order $2$ such that $\lambda_x(x) = x$. We have:
\[ \begin{array}{lll}
S_1 = \{1, 2a+b+d\},  & S_2 = \{1,c\}, & S_3 = \{1, b+c\}, \\
S_4 = \{1, c+d\}, & S_5 = \{1, 2a\}, & S_6 = \{1, 2a +b\}, \\
S_7 = \{1, 2a+b+c\}.& & 
\end{array}
\] (here, $S_5$ is the only left ideal). For subbraces of order $4$, we need those subgroups $H \leq (B,+)$ of order $4$ such that $\delta^{-1}(H)$ is also a subgroup of $\langle e, f, h \rangle$. We find
{\small
\[ \def\arraystretch{1.2}
\begin{array}{llll}
H & \delta^{-1}(H) & H & \delta^{-1}(H) \\ \hline
S_8 = \langle 2a, b+c\rangle   & \langle e^3fh, e^7fh\rangle & S_9 = \langle 2a+b, c\rangle & \langle e^4,h\rangle  \\
S_{10} = \langle b, 2a+c\rangle & \langle e^3f \rangle  & S_{11} = \langle c+d, 2a+b+d\rangle  & \langle efh, e^5fh\rangle  \\
S_{12} = \langle 2a+d,b+c+d\rangle & \langle e^5f \rangle &  S_{13}= \langle d, 2a+b+c\rangle &  \langle e^6h \rangle \\
S_{14} = \langle 2a+b+c, b+d\rangle & \langle e^2 \rangle &  \\ \hline
\end{array}
\]}(here, $S_8$ is the only left ideal). For subbraces of order $8$, we need those subgroups $H \leq (B,+)$ of order $8$ such that $\delta^{-1}(H)\leq\langle e, f, h \rangle$. Thus,
\[ \def\arraystretch{1.2}
\begin{array}{ll}
H & \delta^{-1}(H)  \\ \hline
S_{15} = \langle 2a, b, c \rangle   & \langle e^3fh, e^3f \rangle \\
S_{16} = \langle 2a, b+c, b+d \rangle & \langle e^3fh, e^2\rangle  \\
S_{17} = \langle 2a, b+c, d\rangle & \langle e^3fh, e^6h\rangle  \\ \hline
\end{array}\]
(here, $S_{15}$ is the only left ideal). The only subbrace of order $16$ is the only non-zero proper ideal of $B$, that is, $S_{18} = \langle 2a,b,c,d\rangle$.

The following relations can be easily checked to hold:
\[\begin{array}{r}
S_1, S_4, S_7 \unlhd S_{11} \unlhd S_{15} \unlhd S_{18} \unlhd B\\
S_3, S_5, S_7 \unlhd S_8 \unlhd S_{16} \unlhd S_{18} \unlhd B\\
S_2, S_6, S_7 \unlhd S_{9} \unlhd S_{17} \unlhd S_{18} \unlhd B\\
S_{14} \unlhd S_{15} \unlhd S_{18}\unlhd B\\
S_{12}, S_{13} \unlhd S_{17} \unlhd S_{18}\unlhd B\\
S_{10} \unlhd S_{17} \unlhd S_{18} \unlhd B
\end{array}\]
Therefore, all subraces are subideals but $B$ is not soluble.
\end{aex}

\section{Acknowledgments}

\noindent All authors are members of the non-profit association ``Advances in Group Theory and Applications'' (www.advgrouptheory.com). The third and fifth authors are supported by GNSAGA (INdAM) and wish to express their gratitude to the Department of Mathematics of the University of Valencia for its support during their stay. The third author has been supported by a research visiting grant issued by the Istituto Nazionale di Alta Matematica (INdAM).


\end{document}